[1994/12/01]
\documentclass[10pt, reqno]{amsart}
\usepackage{a4wide}
\usepackage[english, activeacute]{babel}
\usepackage{pifont}
\usepackage{amsmath,amsthm,amsxtra}

\usepackage{epsfig}
\usepackage{amssymb}

\usepackage{latexsym}
\usepackage{amsfonts}

\usepackage{hyperref}
\usepackage{xparse}
\usepackage{tabularx} 
\usepackage{multicol}
\usepackage{color}

\title{Nonlinear stability of Gardner breathers}
\author{Miguel A. Alejo}
\address{Departamento de Matem\'atica, Universidade Federal de Santa Catarina, Brasil}
\email{miguel.alejo@ufsc.br}
\date{\today}

\chardef\bslash=`\\ 

\newtheorem{thm}{Theorem}[section]
\newtheorem{cor}[thm]{Corollary}
\newtheorem{lem}[thm]{Lemma}
\newtheorem{prop}[thm]{Proposition}

\newtheorem{defn}[thm]{Definition}

\theoremstyle{remark}
\newtheorem{rem}{Remark}[section]

\numberwithin{equation}{section}

\newcommand{\R}{\mathbb{R}}

\newcommand{\Z}{\mathbb{Z}}

\newcommand{\la}{\lambda}

\newcommand{\al}{\alpha}
\newcommand{\bt}{\beta}
\newcommand{\ga}{\gamma}

\newcommand{\spawn}{\operatorname{span}}

\newcommand{\ba}{\left( \begin{array}{c}}
\newcommand{\ea}{\end{array}\right)}

 \providecommand{\abs}[1]{\lvert#1 \rvert}

\newcommand{\be}{\begin{equation}}
\newcommand{\ee}{\end{equation}}
\newcommand{\bp}{\begin{proof}}
\newcommand{\ep}{\end{proof}}
\newcommand{\bel}{\begin{equation}\label}
\newcommand{\eeq}{\end{equation}}
\newcommand{\bea}{\begin{eqnarray}}
\newcommand{\eea}{\end{eqnarray}}
\newcommand{\bee}{\begin{eqnarray*}}
\newcommand{\eee}{\end{eqnarray*}}
\newcommand{\ben}{\begin{enumerate}}
\newcommand{\een}{\end{enumerate}}
\newcommand{\nonu}{\nonumber}

\newcommand{\ms}{\medskip}

\newcommand{\eval}[2][\right]{\relax
  \ifx#1\right\relax \left.\fi#2#1\rvert}

\let\abs=\envert

\begin{document}
\begin{abstract}
We show that breather solutions of the Gardner equation, a natural generalization of the KdV and mKdV
equations, are $H^2(\R)$ stable. Through a variational 
approach, we characterize Gardner breathers as minimizers of a new Lyapunov functional and we study the
associated spectral problem, through $(i)$ the analysis of the spectrum of 
explicit linear systems (\emph{spectral stability}), and $(ii)$  controlling  degenerated directions by  using low regularity conservation laws.



\end{abstract}
\maketitle \markboth{Stability of Gardner breathers} {Miguel A. Alejo}
\renewcommand{\sectionmark}[1]{}
\tableofcontents

\section{Introduction}

\subsection{Preliminaries}
 In this paper we consider the nonlinear stability of \emph{breathers} of the Gardner equation
\be\label{GE}
w_t +(w_{xx} +3\mu w^2+w^3)_x=0, \quad \mu\in\R\backslash\{0\},\qquad w(t,x) \in \R, \; (t,x) \in\R^2.
\ee

\ms
\noindent
Specifically, we present here a proof on the  stability in $H^2(\R)$ of Gardner breathers, showing that this stability is 
independent of the value of the parameter $\mu$, which controls the strength of the quadratic nonlinear part or KdV term $w^2$,
in its existence interval for \emph{real} Gardner breathers.
\ms 
\noindent

The Gardner equation \eqref{GE}  is a well-known \emph{completely integrable} model \cite{Ga,AC,La}, with infinitely many conservation laws and well-known 
(long-time) 
asymptotic behavior of its solutions obtained with the help of the inverse scattering transform \cite{GrSl}. As a physical model, \eqref{GE}
describes large-amplitude internal solitary waves, showing a dynamics
which can look rather different from the KdV form. On the other hand, solutions of \eqref{GE} are invariant under space and time translations. 
Indeed, for any $t_0, x_0\in \R$, $w(t-t_0, x-x_0)$ is also a solution. Note that \eqref{GE} is not scaling invariant. 
Moreover, \eqref{GE} is closely related to the modified Korteweg-de Vries (mKdV) equation
\be\label{mKdV} 
u_{t} +(u_{xx} + u^3)_x=0, \qquad u(t,x) \in \R, \;  (t,x) \in\R^2,
\ee 
\noindent
through the search of $L^{\infty}$-solutions. In fact, it is easy to see by substitution, that the following holds:

\begin{prop}\label{connect}
Let $u$ be a  solution of the mKdV equation \eqref{mKdV} with a nonvanishing boundary 
value or condition (NVBC) $\mu\in\R\backslash\{0\}$ at $\pm \infty$. Then $w(t,x):= u(t,x+3\mu^2t) - \mu$ is a  solution of the Gardner equation \eqref{GE}.
\end{prop}

\medskip
The key characteristic of the  Gardner equation \eqref{GE} is that it contains a nonlinear part composed of a Korteweg-de Vries (KdV) quadratic
term $(w^2)_x$ and a positive modified KdV (mKdV) cubic term $(w^3)_x$. The competition between this nonlinear part and the linear  dispersive term
$w_{xxx}$ allows the existence of intricate soliton, multisolitons as well as exact  real-valued \emph{breather} solutions (see \eqref{GEBre}).
A \emph{soliton} is a localized, moving or stationary solution which maintains its form for all time. Similarly, a 
\emph{multi-soliton} is a (not necessarily) explicit solution describing the interaction of several solitons \cite{HIROTA1}.

\medskip
In the case of the Gardner equation \eqref{GE}, the profile of the soliton solution is slightly cumbersome, but it is still given explicitly  by the formula
\be\label{fexplicita2} w(t,x) := Q_{c,\mu} (x-ct), 
\qquad Q_{c,\mu} (s) :=\frac{c}{ \displaystyle{ \mu +\sqrt{\mu^2+\frac{c}{2}}\cosh(\sqrt{c}s) } }.\eeq
By substituting  \eqref{fexplicita2}  into  \eqref{GE}, one has that  $Q_{c,\mu}>0 $ satisfies the nonlinear elliptic equation
\be\label{eqQc}\begin{array}{ll}
Q_{c,\mu}'' -c\, Q_{c,\mu} + 3\mu Q_{c,\mu}^2+Q_{c,\mu}^3=0, \quad Q_{c,\mu}>0, \quad Q_{c,\mu}\in H^1(\R).
\end{array}\ee

This second order, elliptic equation is deeply related to the so-called variational structure of the soliton solution. To be more precise, 
it is well-known that for the Gardner equation the standard conservation laws at the $H^1$-level are the \emph{mass}
\be\label{M1}
M[w](t)  :=  \frac 12 \int_\R w^2(t,x)dx = M[w](0),
\ee 
and \emph{energy} 
\be\label{E2}
E_\mu[w](t)  :=  \frac 12 \int_\R w_x^2 -\mu\int_\R w^3 - \frac 14 \int_\R w^4= E_\mu[w](0),\ee
which is $H^1$-subcritical. For the Gardner equation \eqref{GE}, the Cauchy problem is globally well-posed at such a level of regularity or even better, 
see e.g. Alejo and Kenig-Ponce-Vega \cite{Ale11},\cite{KPV}. Note that these results are not trivial since \eqref{GE} is not scaling
invariant.  
 
\ms

Moreover, for a mKdV solution $u$ with NVBC $\mu$ we also have the natural conservation laws
\be\label{Mnv}
M_{nv}[u](t)  :=  \frac 12 \int_\R (u^2-\mu^2)dx = M_{nv}[u](0),
\ee 
and \emph{energy} 
\be\label{E3}
E[u](t)  :=  \frac 12 \int_\R u_x^2  - \frac 14 \int_\R (u^4-\mu^4)= E[u](0),\ee
which is $H^1$-subcritical. Using these conserved quantities, the variational structure of any Gardner soliton can be characterized as follows: 
there exists a suitable \emph{Lyapunov functional}, \emph{invariant in time} and such 
that the soliton $Q_{c,\mu}$ is a corresponding \emph{extremal point}. Moreover, it is a global minimizer under fixed mass. For the Gardner case, 
this functional is given by (see \cite{Benj} for the mKdV case)
\be\label{H0mk}
\mathcal{H}_0[w](t) = E_\mu[w](t) + c \, M[w](t),
\ee
\noindent
where $c>0$ is the scaling of the solitary wave, and $M[w]$, $E_\mu[w]$ are given in (\ref{M1}) and (\ref{E2}). 
Indeed, it is easy to see that for any $z(t)\in H^1(\R)$ small,
\be\label{Expa1mk}
\mathcal{H}_0[Q_{c,\mu}+z](t)  =  \mathcal{H}_0[Q_{c,\mu}] + \int_\R z(Q_{c,\mu}''-cQ_{c,\mu} +3\mu Q_{c,\mu}^2 +Q_{c,\mu}^3) +  O(\|z(t)\|_{H^1}^2).
\ee
\noindent
The zero order term  above is independent of time, while the first order term in $z$ is zero from (\ref{eqQc}), proving the critical character of 
$Q_{c,\mu}$. 

\medskip

\subsection{Breathers and their stability}
Besides these soliton solutions of the Gardner equation \eqref{GE}, it is possible to find another big set of explicit and oscillatory solutions, 
known in the physical and mathematical literature as the \emph{breather} solution, and which is a periodic in time, spatially localized real function. 
Although there is no universal definition 
for a breather, we will adopt the following convention, that will match the  Gardner, mKdV and also sine-Gordon cases (see \cite{AMP1}).

\begin{defn}[Aperiodic breather]\label{DEF_B}
We say that $B=B(t,x)$ is a breather solution for a particular one-dimensional dispersive equation 
if there are $T>0$ and $L=L(T)\in \R$ such that, for all $t\in \R$ and $x\in \R$, one has 
\be\label{B_def}
B(t+T,x) = B(t,x-L),
\ee
and moreover, the infimum among times $T>0$ such that property \eqref{B_def} is satisfied for such a time $T$ is uniformly positive in space. 
\end{defn}

\begin{rem}
Observe that the last condition ensures that solitons (and multisolitons) are not breathers, since e.g. $Q_{c,\mu}(x-c (t+T)) = Q_{c,\mu}(x-L -ct)$ for
$L:=cT$ but $T$ can be any real-valued time.\footnote{In the case of NLS equations and their solitons, 
Definition \ref{DEF_B} include them because of the $U(1)$ invariance.}
\end{rem}

For the Gardner equation \eqref{GE}, the breather solution  can be obtained by using different methods 
 (e.g. Inverse Scattering, Hirota method, etc), and its expression is characterized by the introduction of the parameter $\mu$ which
controls the quadratic nonlinearity in \eqref{GE}. 
\begin{defn}[Gardner breather]\label{BreatherGE} Let $\al, \bt,\mu \in \R\backslash\{0\}$ such that $\Delta=\al^2+\bt^2-2\mu^2>0$, 
and $x_1,x_2\in \R$. The real-valued  breather solution of the Gardner equation \eqref{GE} is given explicitly by the formula
\be\label{GEBre}B_\mu\equiv B_{\al, \bt,\mu}(t,x;x_1,x_2)   
 :=   2\sqrt{2}\partial_x\Bigg[\arctan\Big(\frac{G_{\al, \bt,\mu}(t,x)}
{F_{\al, \bt,\mu}(t,x)}\Big)\Bigg],
\ee
with $y_1$ and $y_2$ 
\be\label{y1y2GE}
y_1 = x+ \delta t + x_1, \quad y_2 = x+ \ga t + x_2, \quad \delta := \al^2-3\bt^2, \quad  \ga :=3\al^2-\bt^2,
\ee and
\[
\begin{aligned}
G_\mu\equiv G_{\al, \bt,\mu}(t,x) & :=   \frac{\bt\sqrt{\al^2+\bt^2}}{\al\sqrt{\Delta}}\sin(\al y_1) -\frac{\sqrt2 \mu\bt[\cosh(\bt y_2)+\sinh(\bt y_2)]}{\Delta},\\
F_\mu\equiv F_{\al, \bt,\mu}(t,x) & :=   \cosh(\bt y_2)-\frac{\sqrt2\mu\bt[\al\cos(\al y_1)-\bt\sin(\al y_1)]}{\al\sqrt{\al^2+\bt^2}\sqrt{\Delta}}.
\end{aligned}
\]
\end{defn}

\ms
\begin{rem}
This is a four-parametric solution, with two scalings ($\al,\bt$) and two shift translations ($x_1,x_2$). 
Note that we impose $\al^2+\bt^2-2\mu^2>0,$ since we deal with  real-valued solutions. Therefore $\mu\in(0,\mu_{max})$, where

\be\label{mumax}
\mu_{max}:=\sqrt{\frac{\al^2+\bt^2}{2}}.\\
\ee

\medskip
\bigskip

\noindent
Moreover from \eqref{GEBre} one has, for any $k\in \Z$,
\be\label{dege}
 B_{\al,\bt,(-1)^k\mu} (t,x; x_1 + \frac{k\pi}{\al}, x_2)  =  (-1)^k B_{\al,\bt,\mu} (t,x; x_1, x_2),
\ee
which are also solutions of (\ref{GE}). This identity reveals the periodic character of the first translation parameter $x_1$, coupled this time
to the parameter $(-1)^k\mu,~~k\in \Z$.
\end{rem}

\medskip

\begin{rem}
 Note  that we can take the limit when $\al\rightarrow0$ in \eqref{GEBre}, obtaining the so call \emph{double pole} solution for the Gardner equation \eqref{GE}, 
 which it is a natural generalization of the well-known double pole solution of mKdV:
 \end{rem}
\begin{defn}\label{DPoleGE} Let $\bt,\mu \in \R\backslash\{0\}$ such that $\Delta_0=\bt^2-2\mu^2>0$, 
and $x_1,x_2\in \R$. The  real-valued double pole solution of the Gardner equation \eqref{GE} is given explicitly by the formula
\be\label{GEdpole}B_{\bt,\mu}(t,x)   :=   \partial_x \tilde B :=   2\sqrt{2}\partial_x\Big[\arctan\Big(\frac{G_{ \bt,\mu}(t,x)}
{F_{\bt,\mu}(t,x)}\Big)\Big],
\ee
with
\bee
G_{\bt,\mu}(t,x) & := &  \frac{\bt^2y_1}{\sqrt{\Delta_0}} -\frac{\sqrt2 \mu\bt[\cosh(\bt y_2)+\sinh(\bt y_2)]}{\Delta_0},\\
F_{\bt,\mu}(t,x) & := &  \cosh(\bt y_2)-\frac{\sqrt2\mu[1-\bt y_1]}{\sqrt{\Delta_0}},
\eee
and where now $y_1=x-3\bt^2t+x_1$ and $y_2=x-\bt^2t+x_2$.
\end{defn}

\noindent
Note that Gardner breather solutions \eqref{GEBre} are periodic in time, but not in space. 
Additionally, every Gardner  breather satisfies Definition \ref{DEF_B} with $T=\frac{2\pi}{\al(\ga-\delta)}>0$ and $L=-\ga T$. 
A simple but very important remark is that $\delta \neq \ga$, 
for all values of $\al$ and $\beta$ different from zero. This means that the variables  $x+\delta t$ and $x+\ga t$ are {\bf always independent}, 
a property that characterizes breather solutions, which is not satisfied by standard solitons.  On the other hand, $-\ga$ will be for us the 
\emph{velocity} of the breather solution, since it corresponds to the velocity of the \emph{carried hump} in the breather profile. 
Note additionally that breathers have to be considered as \emph{bound states}, since they do not decouple into simple solitons as time evolves. 

\ms

For the Gardner equation, the breather solution was discovered by \cite{PeGr, Ale0}, using the IST. These solutions have become a canonical example  
of  complexity in nonlinear integrable systems \cite{La, AC}. Moreover, 
it is interesting to point out that mKdV and Gardner  breather solutions have also been considered by Kenig, Ponce and Vega and Alejo respectively,
in their proofs of the non-uniform continuity of the mKdV and Gardner flows in the Sobolev spaces $H^s$, $s<\frac 14$ \cite{KPV2, Ale1}. 

\ms

If one studies  perturbations of solitons in \eqref{GE}, \eqref{mKdV} and more general equations,  the concepts of \emph{orbital}, 
and \emph{asymptotic stability} emerge naturally. In particular, since energy and mass are conserved quantities, 
it is natural to expect that solitons are stable in a suitable energy space. Indeed, $H^1$-stability of mKdV and more general 
solitons and multi-solitons has been considered e.g. in Benjamin \cite{Benj}, 
Bona-Souganidis-Strauss \cite{BSS}, Weinstein \cite{We2}, Maddocks-Sachs \cite{MS}, Martel-Merle-Tsai \cite{MMT}, Martel-Merle \cite{MMcol2} 
and Mu\~noz \cite{Mu}. $L^2$-stability of KdV solitons has been proved by Merle-Vega \cite{MV}.
Moreover, asymptotic stability properties for gKdV equations have been studied by Pego-Weinstein \cite{PW} and Martel-Merle
\cite{MMarma,MMnon}, among many other authors. 

\medskip

The underlying question is then the study of the corresponding \emph{stability} of these breather solutions. A first step in that direction
was already done in \cite{AM} (see also \cite{AM0,AMP1,AMP2}), where the nonlinear stability of mKdV breathers was presented. Recent studies
about the stability-instability of these breather structures are \cite{Mu3} and \cite{BPSS}. Other references dealing with similar problems on
stability/instability of coherent structures are  Kowalczyk-Martel-Mu\~noz \cite{KMM1} and \cite{KMM2}, Comech-Cuccagna-Pelinovsky \cite{CoPe},
 Cuccagna-Pelinovsky-Vougalter \cite{CuPeVou},  Grillakis-Shatah-Strauss \cite{GSS}, 
Howard-Zumbrum \cite{HowZu}, Kapitula \cite{Kap1}, \cite{Kap2}, Kapitula-Kevrekidis-Sandstede \cite{KapKeSa}, Kapitula-Promislow \cite{KaPro},
Kaup-Yang \cite{KauYa}, Sandstede \cite{Sand}, Zumbrum \cite{Zu}, Yang \cite{Ya1} and \cite{Ya2}.

\subsection{Main Results}
In this paper, we show a positive answer to the question of the stability of Gardner breathers. In fact, our main result is stated in short as follows:

\ms

\begin{thm}\label{T1p8}
Gardner breathers \eqref{GEBre} are orbitally stable for $H^2$-perturbations, whenever the parameter $\mu\in(0,\mu_{max})$, with $\mu_{max}$ as in \eqref{mumax}.
\end{thm}

\ms

A more precise version of this theorem is given in Theorem \ref{T1gardner}. We will see from the variational characterization of Gardner breathers 
(see Sect.\ref{5}) that the Sobolev  space $H^2$ will arise naturally, since the Lyapunov functional that will have Gardner breathers as minimizers, will
be defined in that space. 


\ms



For the proof of the previous theorem, we will follow some steps introduced in \cite{AM} for mKdV breathers. This approach in the mKdV setting
is far to be trivial when we deal with Gardner breathers,
since many of \emph{at hand} proofs presented for mKdV breathers, are no longer practical as a consequence of the complicated functional form of breather solutions like \eqref{GEBre}. 
Therefore, we have to resort to the deep integrability structure of the \eqref{GE} in order to perform these proofs, by using adapted identities for the
Gardner breather and new nonlinear identities (see Lemma \ref{Id1}). Moreover, we are able to compute explicitly the mass, the energy
and the third conserved quantity defined in $H^2$ of any Gardner breather
solution \eqref{GEBre}, showing explicitly the nonlinear dependence on the parameter $\mu$.

\ms

The main steps of the proof are the following. Firstly, we show that Gardner breathers \eqref{GEBre} satisfy a fourth-order, 
nonlinear ODE. More precisely, every Gardner breather verifies 
\bea\label{BG4thODE}
& J_\mu[B_\mu]:=& B_{\mu,4x}  -   2(\beta^2 -\alpha^2) (B_{\mu,xx} +\mu B_\mu^2 + B_\mu^3) +  
(\alpha^2 + \beta^2)^2 B_\mu + 5 B_\mu B_{\mu,x}^2 + 5B_\mu^2 B_{\mu,xx}
\nonu\\
&&+ \frac 32 B_\mu^5  + 5\mu B_{\mu,x}^2  + 10\mu B_\mu B_{\mu,xx} + 10\mu^2 B_\mu^3 + \frac{15}{2} \mu B_\mu^4  =0.
\eea
\noindent
(cf. equation \eqref{EcBGEfinal}). This result for Gardner breathers is, as far as we know, not present in the literature, althought similar ones were obtained
in \cite{AM,AMP1,AMP2} for the case of mKdV breathers in the line and periodic case and also for sine-Gordon breathers respectively. The proof of this identity 
is based on the close connection between Gardner 
and mKdV with NVBC equations (see Proposition \ref{connect}). That is, we are going to prove that the above identity \eqref{BG4thODE} for Gardner breathers 
is related with an equivalent identity \eqref{GBnvbc} for mKdV breathers with NVBC, that will be easier to be proved, by using the explicit form of the breather 
with NVBC of mKdV, and several new identities related to the structure of the breather. 

\ms

It seems that this equation cannot be obtained from the original approach by Lax \cite{LAX1}, since the dynamics do not decouple in time. 
As far as we know, this is the first time that the previous equation is proved for breathers of the Gardner (and also for mKdV with NVBC \eqref{mKdV})
equation \eqref{GE}. 

\medskip

As second step, we show the variational structure of  Gardner breathers \eqref{GEBre}, 
with the introduction of a new Lyapunov functional (see \eqref{LyapunovGE} for further details)
\ms
\bea\label{LyapunovGEintro}
\mathcal H_\mu[w](t) := F_\mu[w](t) + 2(\bt^2-\al^2)E_\mu[w](t) + (\al^2 +\bt^2)^2 M[w](t),
\eea
\noindent
well-defined in the $H^2$-topology, and for which Gardner breathers are not only \emph{extremal points}, but also \emph{local minimizers}, 
up to  symmetries. This functional will control the perturbative terms and the instability directions arising during of the dynamics, 
the latter as a consequence of the symmetries satisfied by  \eqref{GE}. We will study spectrally the linearized operator around
Gardner breather solutions and we will discover, as in the mKdV case, that it has only one negative eigenvalue. We will prove that by
using new identities which simplify the computation of the Wronskian in the kernel elements of this operator (see Greenberg \cite{Gr} for further details
of the theory  dealing with fourth order eigenvalue problems). This strategy was used first by
Weinstein in \cite{We1}. 



\medskip


\subsection{Organization of this paper}
In Sect.\ref{2} we study generalized Weinstein conditions and we present
some nonlinear identities and stability tests satisfied by Gardner breathers. In Sect.\ref{3} we prove that any 
Gardner and  mKdV (with NVBC) breather solutions satisfy a fourth order nonlinear ODE, which
characterizes them. In Sect.\ref{5} we present a variational characterization of the Gardner breather,
by introducing a new Lyapunov functional which controls the dynamic of small perturbations in the stability problem. 
In Sect.\ref{SpecAn} we focus on the spectral properties of a linear self-adjoint operator related
to the Gardner breather solution. Finally in Sect.\ref{6} we present a detailed version of Theorem \ref{T1p8} and a sketch of its proof.
\bigskip
 \noindent
 
{\bf Acknowledgments.} 
I would like to thank Claudio Mu\~noz and Jos\'e M. Palacios for many richful discussions and
comments on a first version of this paper. I also would like to thank to the Departamento de Ingenier\'ia Matem\'atica (DIM) of U. Chile,
where part of this work was completed, for its kind hospitality and support. Finally I 
would also like to acknowledge insightful comments made by the anonymous referee which improved an earlier version of this work.

\bigskip

\medskip

\section{Nonlinear identities and Weinstein conditions}\label{2}

\medskip

The aim of this section is to get stability tests by computing  generalized Weinstein conditions for any Gardner breather $B_{\mu}$ \eqref{GEBre}. 
We begin with the simpler case of the Gardner soliton \eqref{fexplicita2}, where the mass (\ref{M1}) and the energy (\ref{E2}) are given by the 
quantities\footnote{See \cite[Eqn. 2.443-1-2-3]{GradRyz} for further details in order to compute  these integrals explicitly.}

\be\label{MsolG}
M[Q_{c,\mu}]  :=  2\sqrt{c}-2\sqrt{2}\mu\arctan\Big[\frac{\sqrt{c}}{\sqrt{2}\mu}\Big],
\ee 
\noindent
\text{and}
\be\label{EQc}
E_\mu[Q_{c,\mu}] := -\frac 23 c^{3/2} + 4\mu^2\sqrt{c} - 4\sqrt{2}\mu^3\arctan[\frac{\sqrt{c}}{\sqrt{2}\mu}],
\ee
which reduce,  when $\mu=0$, to the well known mass and energy of the mKdV soliton solution \cite[eqn. (2.1)]{AM}. From these expressions we
see explicitly the coupling between the soliton scaling  $c$ and the parameter $\mu$ in the computed  mass and the energy. Note that 
the Weinstein condition \cite{We2} is now, for $c>0,~\mu\in\R\backslash\{0\}$,
\be\label{WC}
\partial_c M[Q_{c,\mu}] = \frac{\sqrt{c}}{c+2\mu^2}>0.
\ee
This inequality guarantees the nonlinear stability of the Gardner soliton \eqref{fexplicita2}. Moreover note that the same condition for the
energy of the Gardner soliton 
$\partial_c E_\mu[Q_{c,\mu}]=-\frac{c^{3/2}}{c+2\mu^2}$ does not vanish either.

\ms

Now, we approach  the case of Gardner breathers.  Firstly  we present the following identity 
for solutions of the Gardner equation and which will be useful when computing the mass and energy of the breather solution
(see appendix \ref{Id_BilinearGE} for a proof):

\begin{lem}\label{Id_GEq} Let $w(t,x):=\sqrt{2}i\partial_x\log(\frac{F_\mu - iG_\mu}{F_\mu + iG_\mu})$ be any Gardner solution of \eqref{GE}. Then\footnote{Note here that $F_\mu,G_\mu$ are not necessarily 
the same ones introduced in \eqref{GEBre}.} 

\be\label{Id_B2}
w^2 =  2\frac{\partial^2}{\partial x^2}\log(G_\mu^2 + F_\mu^2) - 2\mu w.
\ee
\end{lem}

\ms
\noindent
This result allows us to compute explicitly the mass of any Gardner breather:

\begin{lem}\label{ME} Let $B_{\mu}=B_{\al,\bt,\mu}$ be any Gardner breather, for $\al, \bt,\mu$ as in definition \eqref{GEBre}. 
Then, the mass of $B_\mu$ is
\be\label{MassBG}
 M[B_\mu] := 4\bt + 2\sqrt{2}\mu\arctan\Big[\frac{2\sqrt{2}\mu\bt}{\Delta}\Big].
\ee
\end{lem}
\begin{proof}
It follows directly by using the above identity \eqref{Id_B2} and substitution 
in the definition \eqref{M1}. In fact, we obtain:

\[
M[B_\mu] = \frac{1}{2}\int_{\R}B_\mu^2 = 4\bt + 2\sqrt{2}\mu\arctan\Big[\frac{2\sqrt{2}\mu\bt}{\Delta}\Big] .
\]

\end{proof}

\begin{rem}
Note that as we could expect, after the Gardner soliton case, the mass of any Gardner breather  depends 
on the scalings  $\al,\beta$, and the parameter $\mu$. This dependence slightly differs
from the exclusive dependence of the mass of any mKdV breather on the second scaling parameter $\beta$.
\end{rem}


From the involved integral in \eqref{MassBG}, we can define the \emph{partial} mass of any Gardner breather
as follows:
\begin{defn}\label{MasssBG1} Let $B_{\mu}=B_{\al,\bt,\mu}$ be any Gardner breather, for $\al, \bt,\mu$ as in definition \eqref{GEBre} 
(such that $\Delta=\al^2+\bt^2-2\mu^2>0$). Then, we define the \emph{partial} mass associated to any Gardner breather as:
\begin{align} 
 \mathcal M_{\mu}(t,x)\equiv\mathcal M_{\al,\bt,\mu}(t,x)  &:= \frac 12\int_{-\infty}^x B_{\mu}^2(t,s; x_1,x_2)ds \nonu\\
 &~= 2\bt + \partial_x \log(G_\mu^2 + F_\mu^2)(t,x)-2\sqrt{2}\mu\arctan\Big(\frac{G_\mu(t,x)}{F_\mu(t,x)}\Big).
\end{align}
\end{defn} 

\ms
\noindent

A direct consequence of the above results are the following generalized Weinstein conditions:

\medskip
\begin{cor}\label{WCcor} Let $B_\mu= B_{\al,\bt,\mu}$ be any Gardner breather of the form \eqref{GEBre}. Given $t\in \R$ fixed, let
\be\label{LAB}
\Lambda_\al B_\mu := \partial_\al B_\mu, \quad \hbox{ and }\quad  \Lambda_\bt B_\mu := \partial_\bt B_\mu.
\ee
Then $\forall\mu\in(0,\mu_{max})$  both functions $\Lambda_\al B_\mu$ and $\Lambda_\bt B_\mu$ are in the  Schwartz class for the spatial variable and
they satisfy the identities

\be\label{LAB1a}
\partial_\al M[ B_\mu]  = \int_\R B_\mu \Lambda_\al B_\mu = -\frac{16\mu^2\bt\al}{\Delta^2 + 8\mu^2\bt^2}<0,\\
\ee

and

\be\label{LAB2b}
\partial_\bt M[ B_\mu]   = \int_\R B_\mu \Lambda_\bt B_\mu = 4\Big(\frac{\Delta^2+2\mu^2\Delta +4\mu^2\bt^2}{\Delta^2 + 8\mu^2\bt^2}\Big)>0,\\
%
\ee

\ms
independently  of time. 
\end{cor}
\noindent



\begin{proof}
By simple inspection, one can see that, given $t$ fixed, 
$\Lambda_\al B_{\mu} $ and $\Lambda_\bt B_{\mu} $ are well-defined Schwartz functions. 
The proof of (\ref{LAB1a}) and (\ref{LAB2b}) is consequence of (\ref{MassBG}).
\end{proof}

Consider now the two directions associated to spatial translations. Let $B_{\al,\beta,\mu}$ as introduced in \eqref{GEBre}. We define 

\be\label{B12}
 B_1(t ; x_1,x_2) := \partial_{x_1} B_{\al,\bt,\mu}(t ; x_1, x_2),\quad \hbox{ and } \quad  B_2(t ; x_1,x_2): =\partial_{x_2} B_{\al,\bt,\mu}(t ;  x_1, x_2).
\ee

\ms
\noindent
It is clear that, for all $t\in \R,$ and  $\al,\bt,\mu$ as in definition \ref{BreatherGE} and $x_1,x_2\in \R$, both $B_1$ and $B_2$ are real-valued functions in the  Schwartz class, 
exponentially decreasing in space. Moreover, it is not difficult to see that they are \emph{linearly independent} as functions of 
the $x$-variable, for all time $t$ fixed.

\begin{lem}\label{Id1} Let $B_\mu=B_{\al,\bt,\mu}$ be any Gardner breather of the form \eqref{GEBre}. Then we have
\ben
\item  $B_\mu =\tilde B_{\mu,x}$, with $\tilde B_\mu=\tilde B_{\al,\bt,\mu}$ given by the smooth $L^\infty$-function
\be\label{tBga}
\tilde B_\mu(t,x) := 2\sqrt{2}  \arctan \Big(\frac{G_\mu}{F_\mu}\Big). 
\ee
\item For any fixed $t\in \R$, we have $ (\tilde B_\mu)_t$ well-defined in the Schwartz class, satisfiying
\be\label{2ndga}
B_{\mu,xx} + \tilde B_{\mu,t} + 3\mu B_\mu^2 + B_\mu^3 = 0.
\ee
\item Let $\mathcal M_{\mu}$ be defined by (\ref{MasssBG1}). Then
\be\label{Firstga}
B_{\mu,x}^2  + \frac 12 B_\mu^4 + 2\mu B_\mu^3 + 2B_\mu \tilde B_{\mu,t} - 2(\mathcal M_{\mu})_t=0.
\ee
\item  Finally, let $B_1$ and $B_2$ as in \eqref{B12} and $\tilde B_{i}\equiv\partial_{x_i}\tilde B_\mu,~~~
\tilde B_{ij}\equiv\partial_{x_i}\partial_{x_j}\tilde B_\mu,~i,j=1,2$. Then
\be\label{Idwrons}
\int_{-\infty}^x(\tilde B_{12}^2 -\tilde B_{11}\tilde B_{22}) = -(\mu + B_\mu)\tilde{B}_{11} + \partial^2_{x_1}\partial_x\log\Big(G^2_\mu + F^2_\mu\Big).
\ee
\een
\end{lem}

\begin{proof}
The first item above is a direct consequence of the definition of $B_\mu=B_{\al,\bt,\mu}$ in (\ref{GEBre}). 
On the other hand, (\ref{2ndga}) is a consequence of (\ref{tBga}) and integration in space (from $-\infty$ to $x$) of (\ref{GE}). 
To obtain (\ref{Firstga}) we multiply (\ref{2ndga}) by $B_{\mu,x}$ and integrate in space. Finally to prove \eqref{Idwrons}, 
since we are working with smooth functions, one has $B_\mu =\tilde B_1 + \tilde B_2$, and also
\[
B_1 = \tilde B_{11} + \tilde B_{12}, \qquad\text{and}\qquad B_2 = \tilde B_{12} + \tilde B_{22}.
\]
\noindent
Now, since $\tilde B_{12}^2=B_1^2 + \tilde B_{11}^2 - 2B_1\tilde B_{11}$ and $\tilde B_{22} = B_2 - B_1 + \tilde B_{11}$, we get

\begin{align}\label{inter2}
&(\tilde B_{12}^2 -\tilde B_{11}\tilde B_{22}) = B_1^2 + \tilde B_{11}^2 - 2B_1\tilde B_{11} -\tilde B_{11}(B_2 - B_1 + \tilde B_{11})\nonu \\
&=  B_1^2 - \tilde B_{11}( B_1 +  B_{2}) =  B_1^2 - \tilde B_{11}\partial_x(\tilde B_1 + \tilde B_2) = B_1^2 - \tilde B_{11}B_{\mu,x}.
\end{align}
\noindent
Now integrating \eqref{inter2} in $x$, we obtain:

\begin{align}
&\int_{-\infty}^x(\tilde B_{12}^2 -\tilde B_{11}\tilde B_{22}) =\int_{-\infty}^x(B_1^2 - \tilde B_{11}B_{\mu,x})\\
&= \int_{-\infty}^x B_1^2 -\Big[B_\mu\tilde B_{11}|_{-\infty}^x - \int_{-\infty}^x B_\mu B_{11}\Big].
\end{align}
\noindent
Using that $B_\mu B_{11} = \partial_{x_1}^2(\frac{1}{2}B_\mu^2) - B_1^2$, the last integral above simplifies as follows:

\begin{align}
& \int_{-\infty}^x B_1^2 -\Big[B_\mu\tilde B_{11}|_{-\infty}^x - \int_{-\infty}^x B_\mu B_{11}\Big]\\
&= \int_{-\infty}^x B_1^2 -\Big[B_\mu\tilde B_{11}|_{-\infty}^x - \int_{-\infty}^x (\partial_{x_1}^2(\frac{1}{2}B_\mu^2) - B_1^2)\Big]\\
&=-B_\mu\tilde B_{11}|_{-\infty}^x + \partial_{x_1}^2\int_{-\infty}^x \frac{1}{2}B_\mu^2 = -B_\mu\tilde B_{11} + \partial_{x_1}^2\mathcal{M}_\mu[t,x].
\end{align}

\noindent
Now, remembering \eqref{MasssBG1}, we get
\[
  \partial_{x_1}^2\mathcal{M}_\mu = \partial_{x_1}^2\partial_x \log\Big(G^2_\mu + F^2_\mu\Big) - \mu\tilde B_{11}.
\]

Finally, substituting this derivative in the last equation above, we obtain the desired simplification. The proof is complete.
\end{proof}

\begin{rem}
The reader may compare (\ref{2ndga})-(\ref{Firstga}) with the well known identities for the Gardner soliton solution \eqref{fexplicita2}:
\[
Q_{c,\mu}^{''} -cQ_{c,\mu} + 3\mu Q_{c,\mu}^2 + Q_{c,\mu}^3 =0, \qquad (Q_{c,\mu}^{'})^2 -cQ_{c,\mu}^2 + 2\mu Q_{c,\mu}^3 +\frac 12 Q_{c,\mu}^4 =0.
\]
\end{rem}

We compute now the energy \eqref{E2} of any Gardner breather solution.

\begin{lem}\label{ME2} Let $B_{\mu}=B_{\al,\bt,\mu}$ be any Gardner breather, for $\al, \bt,\mu$ as in definition \eqref{GEBre}. 
Then the energy of $B_{\mu}$ is

\be\label{EnergyB}
E_\mu[B_\mu] :=  \frac{4}{3}\bt\ga + 8\bt\mu^2 + 4\sqrt{2}\mu^3\arctan\Big[\frac{2\sqrt{2}\mu\bt}{\Delta}\Big].
\ee
\end{lem}
\ms
\noindent
\begin{rem}
Note that in the Gardner case, in comparisson with the mKdV case ($\mu=0$),
the sign of the energy is dictated by a nonlinear balance among the velocity $\ga$ and the $\mu$ terms. 
\end{rem}

\begin{proof}
First of all, let us prove the following reduction
\be\label{red2}
E_\mu[B_\mu] (t)= \frac 13\int_\R \Big((\mathcal M_{\mu})_t(t,x)-\mu B_\mu^3(t,x)\Big)dx.
\ee
Indeed,  we multiply (\ref{2ndga}) by $B_\mu$ and integrate in space: we get
\[
\int_\R B_{\mu,x}^2  = \int_\R B_\mu \tilde B_{\mu,t} + 3\mu\int_\R B_\mu^3 + \int_\R B_\mu^4.
\]
On the other hand, integrating (\ref{Firstga}),
\[
\int_\R B_{\mu,x}^2  + \frac 12 \int_\R B_\mu^4 + 2\mu \int_\R B_\mu^3 + 2\int_\R B_\mu \tilde B_{\mu,t} - 2 \int_\R (\mathcal M_{\mu})_t=0.
\]
From these two identities, we get
\[
 \int_\R B_\mu^4 = \frac 43 \int_\R (\mathcal M_{\mu})_t -2\int_\R B_\mu \tilde B_{\mu,t} - \frac{10}{3}\mu \int_\R B_\mu^3,
\]
and therefore
\[
\int_\R B_{\mu,x}^2  = \frac 43 \int_\R (\mathcal M_{\mu})_t -  \int_\R B_\mu \tilde B_{\mu,t} - \frac{\mu}{3}\int_\R B_\mu^3.
\]
Finally, substituting the last two identities into (\ref{E2}), we get (\ref{red2}), as desired.

\medskip

Now we prove (\ref{EnergyB}).  From (\ref{MasssBG1}), we have that
\[
 \mathcal M_{\mu}(t,x)= 2\bt + \partial_x\log(G^2+F^2) -\mu \tilde B_{\mu},
\]
\noindent
and hence,
\[
 (\mathcal M_{\mu})_t(t,x)=  \partial_x\partial_t\log(G^2+F^2) -\mu \tilde B_{\mu,t}.
\]
Now substituting in the energy \eqref{red2}, remembering the identity \eqref{2ndga} and the explicit expression for
 $\mathcal M_\mu[B_\mu]$ in \eqref{MassBG}, we get

\begin{align*}
   E_\mu[B_\mu] (t)&= \frac 13\int_\R \Big((\mathcal M_{\mu})_t(t,x)-\mu B_\mu^3\Big)dx = 
\frac 13\int_\R \Big(\partial_x\partial_t\log(G_\mu^2+F_\mu^2) -\mu \tilde B_{\mu,t}-\mu B_\mu^3\Big)dx\\
&   = \frac 13\int_\R \Big(\partial_x\partial_t\log(G_\mu^2+F_\mu^2) +\mu[B_{\mu,xx}  + 3\mu B_\mu^2 + B_\mu^3]-\mu B_\mu^3\Big)dx\\
&   = \frac 13\int_\R \Big(\partial_x\partial_t\log(G_\mu^2+F_\mu^2) +\mu B_{\mu,xx}  + 3\mu^2 B_\mu^2 \Big)dx\\
&    = \Big(\frac 13\partial_t\log(G_\mu^2+F_\mu^2) +\frac{\mu}{3} B_{\mu,x}\Big)|_{-\infty}^{+\infty}  + \mu^2\int_\R B_\mu^2dx\\
&    = \Big(\frac 13\partial_t\log(G_\mu^2+F_\mu^2) +\frac{\mu}{3} B_{\mu,x}\Big)|_{-\infty}^{+\infty}  + 2\mu^2 M[B_\mu]\\
&    = \Big(\frac 13\partial_t\log(G_\mu^2+F_\mu^2) +\frac{\mu}{3} B_{\mu,x}\Big)|_{-\infty}^{+\infty}  
+ 2\mu^2 (4\bt + 4\sqrt{2}\mu\arctan\Big[\frac{\sqrt{2}\mu\bt}{\Delta}\Big])\\
& = \frac{4}{3}\bt\ga + 8\bt\mu^2 + 4\sqrt{2}\mu^3\arctan\Big[\frac{2\sqrt{2}\mu\bt}{\Delta}\Big].
\end{align*}
\end{proof}


\begin{cor}\label{WCcor2} Let $B_\mu=B_{\al,\bt,\mu}$ be any Gardner breather. Then
\be\label{LAB3}
\partial_\al E_\mu[B_\mu] =8\al\bt(1-\frac{4\mu^4}{\Delta^2+8\mu^2\bt^2}), 
\quad \partial_\bt E_\mu[B_\mu] =4(\al^2-\bt^2) + 8\mu^2\Big(\frac{\Delta^2 +2\mu^2\Delta + 4\mu^2\bt^2}{\Delta^2+8\mu^2\bt^2}\Big).
\ee
\end{cor}

\begin{rem}
Note that the condition $\al=\bt$ is equivalent to the identity 
\[
 \partial_\bt E_\mu[B_\mu] =\frac{8\mu^2\bt^4}{\bt^4+\mu^4},
\] 
\noindent
which is always positive, on the contrary to the mKdV breather case ($\mu=0$). On the other hand the similar identity \eqref{WC}
for the energy of Gardner solitons  can not vanish for any $\mu$.
\end{rem}

\bigskip

\section{Elliptic equations for breathers}\label{3}

\medskip

\subsection{Nonlinear stationary equation for mKdV breathers with NVBC $\mu$}
The objective of this section is to prove that any mKdV breather solution with NVBC $\mu$ satisfies a suitable stationary, elliptic equation.
Indeed, this elliptic equation will be a key step in the proof of an equivalent stationary elliptic equation  for any Gardner breather, which
it is a direct consequence of the close connection between mKdV and Gardner solutions, as showed in Proposition \ref{connect}. 
First and for the sake of completness, 
we present these kind of mKdV breathers with NVBC\footnote{See \cite{Ale0} for further details.}:


\begin{defn}\label{BmksdNVBC} Let $\al, \bt,\mu \in \R\backslash\{0\}$ such that $\Delta=\al^2+\bt^2-2\mu^2>0$, 
and $x_1,x_2\in \R$. The real-valued breather solution of the mKdV equation \eqref{mKdV} with NVBC $\mu$ at $\pm\infty$
is given explicitly by the formula
\be\label{Bmkdv}B\equiv B^{nv}_{\al, \bt,\mu}(t,x;x_1,x_2)   
 := \mu +  2\sqrt{2}\partial_x\Bigg[\arctan\Big(\frac{g(t,x)}{f(t,x)}\Big)\Bigg],
\ee
with 

\[
\begin{aligned}
&g(t,x):= G_\mu(t,x-3\mu^2t),\\
&f(t,x):= F_\mu(t,x-3\mu^2t),
\end{aligned}
\]
\noindent
and $G_\mu,F_\mu$ defined in \eqref{GEBre}.
\end{defn}

\begin{rem}
Note that every mKdV breather $B$ (with or without NVBC) satisfies Definition \ref{DEF_B} with the same selection for $L$ and $T$ parameters that
Gardner breathers $B_\mu$ \eqref{GEBre}.
\end{rem}

\ms
\noindent
Let $B$ any  mKdV breather with NVBC. When we compute its mass \eqref{Mnv} (see appendix \ref{Id_BilinearGE} for 
a complementary proof of this identity), we obtain
\be\label{massBnvbc}
M_{nv}[B]:=\frac{1}{2}\int_\R(B^2 - \mu^2)dx = \partial_x\log(f^2 + g^2)|_{-\infty}^{+\infty} = 4\bt.
\ee
\noindent
We can also define the \emph{partial} mass associated to a mKdV breather with NVBC in the following way

\be\label{partialmassBnvbc}
\mathcal M_{nv}[B]:=\frac{1}{2}\int_{-\infty}^x(B^2 - \mu^2)dx = 2\bt + \partial_x\log(f^2 + g^2)(t,x).
\ee
\noindent
Note that the mass of any mKdV breather with NVBC \eqref{massBnvbc} is indeed equal to the mass of the well known mKdV breather solution (see \cite{AM}). 
Now, we show the following identities for mKdV breathers with NVBC:

\begin{lem}\label{Id1mkdv} Let $B=B^{nv}_{\al,\bt,\mu}$ be any mKdV breather with NVBC  of the form \eqref{BmksdNVBC}. Then we have
\ben
\item  $B =\tilde B_{x}$, with $\tilde B=\tilde B^{nv}_{\al,\bt,\mu}$ given by 
\be\label{tBmkdv}
\tilde B(t,x) := \mu x + 2\sqrt{2}  \arctan \Big(\frac{g}{f}\Big). 
\ee
\item For any fixed $t\in \R$, we have $ (\tilde B)_t$ well-defined in the Schwartz class, satisfiying
\be\label{2ndmkdv}
B_{xx} + \tilde B_{t} + B^3 -\mu^3= 0.
\ee
\item Finally, let $\mathcal M_{nv}$ be defined by (\ref{partialmassBnvbc}). Then
\be\label{Firstmkdv}
B_{x}^2  + \frac 12 B^4 + 2B \tilde B_{t} - 2(\mathcal M_{nv})_t - 2\mu^3B + \frac{3}{2}\mu^4=0.
\ee
\een
\end{lem}

\begin{proof}
The first item above is a direct consequence of the definition of $B$ in (\ref{BmksdNVBC}). 
On the other hand, (\ref{2ndmkdv}) is a consequence of (\ref{tBmkdv}) and integration in space (from $-\infty$ to $x$) of (\ref{BmksdNVBC}). 
Finally, to obtain (\ref{Firstmkdv}) we multiply (\ref{2ndmkdv}) by $B_{x}$ and integrate in space taking into account the NVBC $\mu$ at $\pm\infty$. 
\end{proof}

The next nontrivial identity for  mKdV breathers with NVBC \eqref{BmksdNVBC} will be useful in the proof of the nonlinear stationary equation that they satisfy.

\begin{lem}\label{Iddificil}
Let $B=B^{nv}_{\al,\bt,\mu}$ be any mKdV breather with NVBC  \eqref{BmksdNVBC}. Then, for all $t\in \R$, 
\be\label{ide1nvbc}
 B_{xt} +  2(\mathcal M_{nv})_t B   = \Big(2(\bt^2 -\al^2)+5\mu^2\Big)\tilde B_t  + 
 \Big(\al^2 +\bt^2)^2 +6\mu^2(\bt^2-\al^2 + \frac{3}{2}\mu^2)\Big)(B-\mu).
\ee
\end{lem}
\begin{proof} 
See appendix \ref{apIddificil} for a detailed proof of this nonlinear identity.

\end{proof}

\begin{prop}\label{GBnvbc} Let $B= B^{nv}_{\al,\bt,\mu}$ be any mKdV breather with NVBC \eqref{BmksdNVBC}. 
Then, for any  fixed $t\in \R$, $B$ satisfies the nonlinear stationary equation  
\bea\label{EcBnvbc}
&J[B]:=& B_{(4x)} -\Big(2(\bt^2 -\al^2)+5\mu^2\Big) (B_{xx} + B^3)  +\Big((\al^2 +\bt^2)^2 +6\mu^2(\bt^2-\al^2 + \frac{5}{4}\mu^2) \Big)B \nonu\\
&& + 5 BB_x^2 + 5B^2 B_{xx} + \frac 32 B^5 - 4(\bt^2 -\al^2+\mu^2)\mu^3 - (\al^2 +\bt^2)^2\mu=0.
\eea
\end{prop}

\begin{proof}
From (\ref{2ndmkdv}) and (\ref{Firstmkdv}), one has
\begin{align*}  
& J[B] = -( \tilde B_t + B^3-\mu^3)_{xx} + (2(\bt^2 -\al^2)+5\mu^2) (\tilde B_t-\mu^3)  +
((\al^2 +\bt^2)^2 +6\mu^2(\bt^2-\al^2+\frac{5}{4}))B \nonu\\
& + 5 BB_x^2  + 5B^2 B_{xx} + \frac 32 B^5 - 4(\bt^2 -\al^2+\mu^2) - (\al^2 +\bt^2)^2\mu\nonu\\
& =   - B_{tx} - BB_x^2 + 2B^2B_{xx} + (2(\bt^2 -\al^2)+5\mu^2) (\tilde B_t-\mu^3)  +((\al^2 +\bt^2)^2 +6\mu^2(\bt^2-\al^2+\frac{5}{4}))B \nonu\\
&+ \frac 32 B^5 - 4(\bt^2 -\al^2+\mu^2) - (\al^2 +\bt^2)^2\mu\nonu\\
& = - B_{tx} + B \Big[ \frac 12 B^4 + 2B \tilde B_t -2 (\mathcal M_{nv})_t-2\mu^3B+\frac{3}{2}\mu^4 \Big]  -2B^2 ( \tilde B_t + B^3-\mu^3) + \frac 32 B^5 \nonu\\
&- 4(\bt^2 -\al^2+\mu^2) - (\al^2 +\bt^2)^2\mu\nonu\\
&  \qquad + (2(\bt^2 -\al^2) +5\mu^2)(\tilde B_t-\mu^3)  +((\al^2 +\bt^2)^2 +6\mu^2(\bt^2-\al^2+\frac{5}{4}))B\nonu\\
&= - [ B_{tx} + 2 (\mathcal M_{nv})_t B ] + [2(\bt^2 -\al^2)+5\mu^2]\tilde B_t  + [(\al^2 +\bt^2)^2 +6\mu^2(\bt^2-\al^2 + \frac{3}{2}\mu^2)]B\nonu\\
&-6(\bt^2-\al^2 + 3/2\mu^2)\mu^3 - (\al^2+\bt^2)^2\mu\nonu\\
& = - [ B_{tx} + 2 (\mathcal M_{nv})_t B ] + [2(\bt^2 -\al^2)+5\mu^2]\tilde B_t  + [(\al^2 +\bt^2)^2 +6\mu^2(\bt^2-\al^2 + \frac{3}{2}\mu^2)](B-\mu) = 0.
\end{align*}
In the last line we have used (\ref{ide1nvbc}).
\end{proof}

\subsection{Nonlinear stationary equation for Gardner breathers}
Our aim in this section is to prove that any Gardner breather solution $B_\mu$ satisfies a suitable stationary, elliptic equation, by using the close 
relation  with mKdV breathers with NVBC, as presented in Proposition \eqref{connect}.

\begin{thm}\label{GBGE} Let $B_\mu\equiv B_{\al,\bt,\mu}$ be any Gardner breather \eqref{GEBre}. Then, for any  fixed $t\in \R$, $B_\mu$ 
satisfies the nonlinear stationary equation  
\bea\label{EcBGEfinal}
& J_\mu[B_\mu]:=& B_{\mu,4x}  -   2(\beta^2 -\alpha^2) (B_{\mu,xx} + 3\mu B_\mu^2 + B_\mu^3) +  
(\alpha^2 + \beta^2)^2 B_\mu + 5 B_\mu B_{\mu,x}^2 + 5B_\mu^2 B_{\mu,xx}
\nonu\\
&&+ \frac 32 B_\mu^5  + 5B_{\mu,x}^2  + 10\mu B_\mu B_{\mu,xx} + 10\mu^2B_\mu^3 + \frac{15}{2} \mu B_\mu^4  =0.
\eea
\end{thm}

\noindent

\begin{rem}
This identity can be seen as the \emph{nonlinear stationary equation} satisfied by any Gardner breather profile \eqref{GEBre}, 
and therefore it is independent of time and translation parameters $x_1,x_2\in \R$. One can compare with the Gardner soliton solution
$Q_{c,\mu}(x- ct -x_0)$, which satisfies the standard elliptic equation (\ref{eqQc}), obtained as the first variation of the $H^1$ 
Weinstein functional (\ref{H0mk}). Moreover note that \eqref{EcBGEfinal} and \eqref{GEBre} reduces to \cite[eqn. (5.2)]{AM} and \eqref{mKdV} 
(up to translations) when $\mu=0$.
\end{rem}

\begin{proof}
From Proposition \ref{connect}, we first rewrite our Gardner breather $B_\mu$ as $B_\mu = B - \mu$, where here $B$ is a mKdV breather solution 
with NVBC $\mu$ at $\pm\infty$ as we presented in \eqref{BmksdNVBC}. Hence, substituting $B_\mu = B - \mu$ in \eqref{EcBGEfinal} we obtain:

\begin{align}
& J_\mu[B-\mu]\\
&= B_{4x} - 2(\bt^2 -\al^2)\Big(B_{xx} + 3\mu(B-\mu)^2 + (B-\mu)^3\Big) + (\al^2 +\bt^2)^2 (B-\mu)+5(B-\mu)B_x^2\nonu\\
&  + 5(B-\mu)^2B_{xx} + \frac{3}{2}(B-\mu)^5 + 5\mu B_x^2 + 10\mu(B-\mu)B_{xx} + 10\mu^2(B-\mu)^3 + \frac{15}{2}\mu(B-\mu)^4\nonu\\
& = B_{4x} - 2(\beta^2 -\alpha^2) (B_{xx} + B^3)+(\al^2 + \bt^2)^2 B + 5 BB_x^2 + 5B^2 B_{xx} + \frac 32 B^5 - 2(\beta^2 -\alpha^2)\mu\nonu\\
&- 2(\bt^2-\al^2)(2\mu^3-3\mu^2B)  + 5\mu^2B_{xx} - \frac{15}{2}\mu B^4 +15\mu^2B^3 - 15\mu^3B^2 + \frac{15}{2}\mu^4B - \frac{3}{2}\mu^5-10\mu^5 \nonu\\
&-10\mu^2B_{xx} + 10\mu^2B^3-30\mu^3B^2 +30\mu^4B  + \frac{15}{2}\mu B^4 -30\mu^2B^3 + 45\mu^3B^2 -30\mu^4B + \frac{15}{2}\mu^5\nonu\\
& = B_{4x} - 2(\beta^2 -\alpha^2) (B_{xx} + B^3)+(\al^2 + \bt^2)^2 B + 5 BB_x^2 + 5B^2 B_{xx} + \frac 32 B^5  - 5\mu^2(B_{xx} + B^3)\nonu\\
& + 6\mu^2(\bt^2 - \al^2 +\frac{5}{4}\mu^2) B -4\mu^5 - 4(\bt^2 - \al^2)\mu^3 - (\al^2 +\bt^2)^2\mu\nonu\\
& = B_{(4x)} -\Big(2(\bt^2 -\al^2)+5\mu^2\Big) (B_{xx} + B^3)  +\Big((\al^2 +\bt^2)^2 +6\mu^2(\bt^2-\al^2 + \frac{5}{4}\mu^2) \Big)B \nonu\\
& + 5 BB_x^2 + 5B^2 B_{xx} + \frac 32 B^5 - 4(\bt^2 -\al^2+\mu^2)\mu^3 - (\al^2 +\bt^2)^2\mu=J[B]=0\nonu,
\end{align}
where in the last line we have used \eqref{EcBnvbc}.
\end{proof}

\ms
Although the shift parameters $x_1,x_2$ are chosen independents of time, a simple argument ensures that the previous Theorem still holds under time dependent,
translation parameters $x_1(t)$ and $x_2(t)$.

\begin{cor}\label{Cor32}
Let $B^0_{\al,\bt,\mu} = B^0_{\al,\bt,\mu}(t,x; 0,0)$ be any Gardner breather as in \eqref{GEBre}, 
and $x_1(t),$ $x_2(t)\in \R$ two continuous functions, defined for all $t$ in a given interval. Consider the modified breather
\[
 B_{\al,\bt,\mu} (t,x):=B^0_{\al,\bt,\mu}(t,x; x_1(t),x_2(t)), \qquad (\hbox{cf. \eqref{GEBre}}). 
\]
Then $B_{\al,\bt,\mu} $ satisfies \eqref{EcBGEfinal}, for all $t$ in the considered interval.
\end{cor}

\begin{proof}
A direct consequence of the invariance of the equation (\ref{EcBGEfinal}) under spatial translations. 
\end{proof}

\medskip

\section{Variational characterization of Gardner breathers}\label{5}

\medskip
In this section we introduce a new $H^2$-Lyapunov functional for the Gardner equation (\ref{GE}). 
Consider  $w_0\in H^2(\R)$ and let $w=w(t) \in H^2(\R)$ be the associated local 
in time solution of the Cauchy problem associated to (\ref{GE}), with initial condition $w(0)=w_0$ (cf. \cite{KPV}). 
We first define the $H^2$-functional 

\be\label{F1ga}
\begin{aligned}
 F_\mu[w](t)  &  :=   \frac 12 \int_\R w_{xx}^2dx - 5\mu\int_\R ww^2_xdx + \frac{5}{2}\mu^2\int_\R w^4dx 
 - \frac{5}{2} \int_\R w^2w_x^2dx+ \frac{3}{2}\mu \int_\R w^5+ \frac{1}{4} \int_\R w^6dx.
\end{aligned}
\ee

\noindent
\begin{lem}\label{dF10} Let $w$ be a local (in time) $H^2$-solution of the the Gardner equation \eqref{GE} with initial data $w_0$. Then 
the functional $F_\mu[w](t)$ \eqref{F1ga} is a conserved quantity. Moreover, $w$ is a global (in time) $H^2$-solution.
\end{lem}

The existence of this last conserved quantity is a deep consequence of the
\emph{integrability property} for the Gardner equation. In particular,  it is not present in a general, non-integrable gKdV equation. 
The proof of Lemma \ref{dF10} is an easy computation. Moreover, as we did for the mass and the energy \eqref{M1}-\eqref{E2}, 
we  are also able  to get an explicit formula of $F_\mu$ at the Gardner breather $B_\mu$, by using
the following relations between breathers and solitons which are valid for the Gardner (and mKdV) equation
 
 \be\label{r1mass}M[B_\mu] = 2 Re \Big[M[Q_{c,\mu}]|_{\sqrt{c}=\bt+i\al}\Big]\qquad\text{and}\qquad E_\mu[B_\mu] = 2 Re\Big[ E_\mu[Q_{c,\mu}]|_{\sqrt{c}=\bt+i\al}\Big].\ee
 \noindent
 \ms
 
\begin{lem}\label{FBre} Let $B_{\mu}=B_{\al,\bt,\mu}$ be any Gardner breather, for $\al, \bt,\mu$ as in definition \eqref{GEBre}. 
Then we have that\\
\be\label{FBreG}
 F_\mu[B_\mu] := \frac{4}{15}\Big(3\bt(\bt^4-10\bt^2\al^2+5\al^4) -10\mu^2\bt(\bt^2-3\al^2-6\mu^2) \Big)
 + 8\sqrt{2}\mu^5\arctan\Big[\frac{2\sqrt{2}\mu\bt}{\Delta}\Big].
\ee
\end{lem}
\ms
\begin{proof}
 Firstly,  integrating directly \eqref{F1ga} we get an expression for $F_\mu$ at $Q_{c,\mu}$:
 
 \be\label{r3FQ}F_\mu[Q_{c,\mu}] = \frac{2}{15}\sqrt{c}\Big(3c^2 - 10\mu^2c + 60\mu^4\Big) -8\sqrt{2}\mu^5\arctan(\frac{\sqrt{c}}{\sqrt{2}\mu}).\ee
 \ms
 \noindent
Now, applying the same strategy used for relations \eqref{r1mass}, we get
 \be\label{r4F}F_\mu[B_\mu] = 2 Re\Big[ F_\mu[Q_{c,\mu}]|_{\sqrt{c}=\bt+i\al}\Big]=\eqref{FBreG}.\ee
 \end{proof}


Using the functional $F_\mu[w]$ \eqref{F1ga}, we build a new $H^2$-Lyapunov functional specifically associated to the breather solution. 
Let $B_\mu = B_{\al,\bt,\mu}$ be any Gardner breather, and $t\in \R$, and $M[w]$ and $E_\mu[w]$ given in \eqref{M1} and \eqref{E2} respectively. We define

\ms
\bea\label{LyapunovGE}
\mathcal H_\mu[w](t) := F_\mu[w](t) + 2(\bt^2-\al^2)E_\mu[w](t) + (\al^2 +\bt^2)^2 M[w](t).
\eea
\ms
\noindent

Therefore, $\mathcal H_\mu[w]$ is  a real-valued conserved quantity, well-defined for $H^2$-solutions of (\ref{GE}). 
Note additionally that the functionals $\mathcal{H}_{\mu=0}$ and $\mathcal H_\mu$ for the mKdV and Gardner equations are surprisingly the same.

\begin{lem}\label{LyaBG} Let $B_{\mu}=B_{\al,\bt,\mu}$ be any Gardner breather, for $\al, \bt,\mu$ as in definition \eqref{GEBre}. 
Then we have that
\bea\label{LyapunovBreG0}
\mathcal H_\mu[B_\mu] := h_1 + h_22\sqrt{2}\mu\arctan\Big[\frac{2\sqrt{2}\mu\bt}{\Delta}\Big],\eea

\bea\label{LyapunovBreG00}h_1=\frac{8\bt}{15}\Big(4\bt^4 + 20\al^2\bt^2+ 5\mu^2(5\bt^2-3\al^2+6\mu^2)\Big),\quad 
h_2=\Big((\al^2+\bt^2)^2 + 4\mu^2(\bt^2-\al^2 + \mu^2)\Big).
\eea

\end{lem}
\begin{proof}
Collecting the mass \eqref{MassBG}, the energy \eqref{EnergyB} and $F_\mu$ at $B_\mu$ \eqref{FBreG}, 
and substituting at \eqref{LyapunovGE}, we get \eqref{LyapunovBreG00}.
\end{proof}

\ms
\noindent
Moreover, one has the following

\begin{lem}\label{crit}
Gardner breathers $B_\mu$ \eqref{GEBre} are critical points of the Lyapunov functional $\mathcal H_\mu$ \eqref{LyapunovGE}.
In fact, for any $z\in H^2(\R)$  with sufficiently small $H^2$-norm, and $B_\mu=B_{\al,\bt,\mu}$  any Gardner breather solution,
then, for all $t\in \R$,  one has
\be\label{EE}
\mathcal{H}_\mu[B_\mu +z] - \mathcal{H}_\mu[B_\mu]  = \frac 12\mathcal Q_\mu[z] + \mathcal N_\mu[z],
\ee
with $\mathcal Q_\mu$ being the quadratic form defined in \eqref{Qmu}, and $\mathcal N_\mu[z]$ satisfying $|\mathcal N_\mu[z] | \leq K\|z\|_{H^2(\R)}^3.$
\end{lem}
\begin{proof}
We compute:
\begin{align*}
&\mathcal{H}_\mu[B_\mu+z] =  \frac 12 \int_\R (B_\mu+z)_{xx}^2 -\frac 52 \int_\R (B_\mu+z)^2(B_\mu+z)_x^2 + \frac 14 \int_\R (B_\mu+z)^6\\
& - 5 \mu\int_\R(B_\mu+z)(B_\mu+z)_x^2 +\frac{3}{2}\mu\int_\R (B_\mu+z)^5 + \frac{5}{2}\mu^2\int_\R(B_\mu+z)^4 + (\bt^2-\al^2) \int_\R (B_\mu+z)_x^2 \\
&- \frac 12 (\bt^2 -\al^2) \int_\R (B_\mu+z)^4 - 2 \mu (\bt^2 -\al^2) \int_\R (B_\mu+z)^3 +\frac 12 (\al^2 +\bt^2)^2 \int_\R (B_\mu+z)^2
\end{align*}
\begin{align*}
& =  \frac 12 \int_\R B_{\mu,xx}^2 -\frac 52 \int_\R B_\mu^2B_{\mu,x}^2 + \frac 14 \int_\R B_\mu^6 -5\mu\int_\R B_\mu B_{\mu,x}^2
  + \frac 32 \mu\int_R B_\mu^5 + \frac{5}{2}\mu^2\int_\R B_\mu^4\\
&  + (\bt^2-\al^2) \int_\R B_{\mu,x}^2 - \frac 12 (\bt^2 -\al^2) \int_\R B_\mu^4
-2(\bt^2 -\al^2)\mu\int_\R B_\mu^3 +\frac 12 (\al^2 +\bt^2)^2 \int_\R B_\mu^2\\
&  + \int_\R \big[ B_{\mu,4x}  -   2(\beta^2 -\alpha^2) (B_{\mu,xx} +3\mu B_\mu^2 + B_\mu^3) + (\alpha^2 + \beta^2)^2 B_\mu + 5 B_\mu B_{\mu,x}^2\\
&   + 5B_\mu^2 B_{\mu,xx} + \frac 32 B_\mu^5  + 5B_{\mu,x}^2  + 10\mu B_\mu B_{\mu,xx} 
+ 10\mu^2 B_\mu^3 + \frac{15}{2} \mu B_\mu^4 \big] z\\
&  +\frac 12 \Big[  \int_\R z_{xx}^2 + 2(\bt^2 -\al^2)\int_\R z_{x}^2 +(\al^2 +\bt^2)^2\int_\R z^2 - 5\int_\R B_\mu^2 z_x^2\\
&  - 10\mu\int_\R B_{\mu}z_{x}^2  + 10\int_\R B_\mu B_{\mu,x} z_xz+\int_\R(5B_{\mu,x}^2 +10B_\mu B_{\mu,xx} +\frac{15}{2} B_\mu^4\\
&  -6(\beta^2-\al^2)B_\mu^2)z^2  + 3\mu\int_\R(10 B_\mu^3- 4(\beta^2-\al^2)B_\mu + \frac{10}{3}B_{\mu,xx} + 10\mu B_\mu^2)z^2\Big]\\
&  -\frac52\int_\R( z^2z_x^2 + 2B_{\mu,x}z^2z_x + 2B_\mu zz_x^2) + \int_\R 5B_\mu^3z^3 + \frac{15}{4}B_\mu^2z^4 +\frac 32 \int_\R B_\mu z^5 +\frac 14 \int_\R z^6\\
& -5\mu\int_\R zz_x^2 +15\mu\int_\R B_\mu^2z^3 + \frac{15}{2}\mu \int_\R B_\mu z^4 +\frac{3\mu}{2}\int_\R z^5 + 10\mu^2\int_\R B_\mu z^3\\
& +\frac{5}{2}\mu^2\int_\R z^4 -2(\bt^2-\al^2)\int_\R B_\mu z^3  -\frac 12 (\bt^2-\al^2)\int_\R z^4  -2\mu(\bt^2-\al^2)\int_\R z^3.
\end{align*}
We finally obtain:

\[
\mathcal{H}_\mu[B_\mu+z]   =  \mathcal{H_\mu}[B_\mu] + \int_\R J_\mu[B_\mu] z(t)  + \frac 12\mathcal Q_\mu[z] + \mathcal N_\mu[z],
\]
where the quadratic form $\mathcal Q_\mu$, associated to the linearized operator $\mathcal L_\mu$ \eqref{LGE}, is defined in the following way:  
\be\label{Qmu}
\begin{aligned}
\mathcal Q_\mu[z] & := \int_\R z \mathcal L_\mu[z] =   \int_\R z_{xx}^2 + 2(\bt^2 -\al^2)\int_\R z_{x}^2 +(\al^2 +\bt^2)^2\int_\R z^2 
- 5\int_\R B_\mu^2 z_x^2 - 10\mu\int_\R B_{\mu}z_{x}^2  \\
&\quad +10\int_\R B_\mu B_{\mu,x} z_xz+\int_\R(5B_{\mu,x}^2 +10B_\mu B_{\mu,xx} +\frac{15}{2} B_\mu^4 -6(\beta^2-\al^2)B_\mu^2)z^2\\
&\quad  + 10\mu\int_\R B_{\mu,x}z_{x}z + 3\mu\int_\R(10 B_\mu^3 - 4(\beta^2-\al^2)B_\mu + \frac{10}{3}B_{\mu,xx}+ 10\mu B_\mu^2)z^2.
\end{aligned}
\ee
\noindent
Note that, from Theorem \ref{GBGE}, one has $J_\mu[B_\mu] \equiv 0$. Finally, the term $N_\mu[z]$ is given by 
\begin{align*}
\mathcal N_\mu[z]  := &  -\frac52\int_\R( z^2z_x^2 + 2B_{\mu,x}z^2z_x + 2B_\mu zz_x^2) + \int_\R 5B_\mu^3z^3 + \frac{15}{4}B_\mu^2z^4\\
& +\frac 32 \int_\R B_\mu z^5 +\frac 14 \int_\R z^6 -5\mu\int_\R zz_x^2 +15\mu\int_\R B_\mu^2z^3 + \frac{15}{2}\mu \int_\R B_\mu z^4\\
& +\frac{3\mu}{2}\int_\R z^5 + 10\mu^2\int_\R B_\mu z^3 +\frac{5}{2}\mu^2\int_\R z^4 -2(\bt^2-\al^2)\int_\R B_\mu z^3 \\
& - \frac 12 (\bt^2-\al^2)\int_\R z^4  - 2\mu(\bt^2-\al^2)\int_\R z^3.
\end{align*}
Therefore, from direct estimates one has $\mathcal N_\mu[z] = O(\|z\|_{H^2(\R)}^3),$ as desired.
\end{proof}

\section{Spectral properties around Gardner breathers}\label{SpecAn}
Let $z\in H^4(\R)$, and $B_{\mu}$ be any Gardner breather, with shift parameters $x_1,x_2$. We define  $\mathcal L_\mu$ as the linearized operator 
associated to $B_\mu$, i.e. the bilinear operator obtained after a
 linearization of the Lyapunov functional \eqref{LyapunovGE}  at the Gardner breather $B_{\mu}$, as follows: 
\be\label{LGE}
\begin{aligned}
\mathcal L_\mu[z] & := z_{(4x)} -2(\beta^2-\al^2) z_{xx} +(\al^2+\beta^2)^2 z \\
& \quad +5B_\mu^2 z_{xx} +10B_\mu B_{\mu,x} z_x +(5B_{\mu,x}^2 +10B_\mu B_{\mu,xx} +\frac{15}{2} B_\mu^4 -6(\beta^2-\al^2)B_\mu^2)z\\
& \quad + 10\mu B_{\mu}z_{xx} + 10\mu B_{\mu,x}z_{x} + 3\mu\Big[10 B_\mu^3 - 4(\beta^2-\al^2)B_\mu + \frac{10}{3}B_{\mu,xx}
+ 10\mu B_\mu^2\Big]z.
\end{aligned}
\ee

The following concept, associated to a bilinear operator like $\mathcal Q_\mu$ \eqref{Qmu},  is standard and it will be useful for us.

\begin{defn}\label{Neg_dir}
Any nonzero  $z \in H^2$ is said to be a positive (null, negative) direction for $\mathcal Q_\mu$ if we have $\mathcal Q_\mu[z] >0$ $(= 0,< 0)$.
\end{defn}

Essential for the proof of the main result of this work, Theorem \ref{T1p8}, is the spectral study of the associated linear operator $\mathcal L_\mu$ appearing from Theorem \ref{GBGE} and 
particularly, equation \eqref{EcBGEfinal}. Hence, in this section we describe the spectrum of this operator. More precisely, 
our main purpose is to find a suitable coercivity property, 
independently of the nature of scaling parameters. The main result of this section is contained in Proposition \ref{PropOrtog}. Part of the 
analysis carried out in this section has been previously introduced for solitons by Lax \cite{LAX1},  Maddocks-Sachs \cite{MS} and for mKdV breathers by
Alejo-Mu\~noz \cite{AM}, so we follow their arguments adapted to the Gardner breather case, sketching several proofs.  %

\begin{lem} $\mathcal L_\mu$ is a linear, unbounded operator in $L^2(\R)$, with dense domain $H^4(\R)$. Moreover, $\mathcal L_\mu$ is self-adjoint. 
\end{lem}


From standard spectral theory of unbounded operators with rapidly decaying coefficients, it is enough to prove that $\mathcal L_\mu^* =\mathcal L_\mu$ in $H^4(\R)$. 
\begin{proof}
Let $z,w\in H^4(\R)$. Integrating by parts, one has
\begin{align*}
&\int_\R w \mathcal L_\mu[z]  = \int_\R  w\big[ z_{(4x)} -2(\beta^2-\al^2) z_{xx} +(\al^2+\beta^2)^2 z+5B_\mu^2 z_{xx} +10B_\mu B_{\mu,x} z_x \big] \\
&    + \int_\R \big[5B_{\mu,x}^2 +10B_\mu B_{\mu,xx} +\frac{15}{2} B_\mu^4 -6(\beta^2-\al^2)B_\mu^2 \big] z w\\
&    + \int_\R \big[10\mu B_{\mu}z_{xx} + 10\mu B_{\mu,x}z_{x} + 3\mu(10 B_\mu^3 - 4(\beta^2-\al^2)B_\mu + \frac{10}{3}B_{\mu,xx}+ 10\mu B_\mu^2)z\big]w\\
& =\int_\R  z\big[ w_{(4x)} -2(\beta^2-\al^2) w_{xx} +(\al^2+\beta^2)^2 w+5B_\mu^2 w_{xx} +10B_\mu B_{\mu,x} w_x\big]  \\
&    + \int_\R z\big[5B_{\mu,x}^2w +10B_\mu B_{\mu,xx}w +\frac{15}{2} B_\mu^4w -6(\beta^2-\al^2)B_\mu^2w \big]  \\
&    + \int_\R z\big[10\mu B_{\mu}w_{xx} + 10\mu B_{\mu,x}w_{x} + 3\mu(10 B_\mu^3 - 4(\beta^2-\al^2)B_\mu  + \frac{10}{3}B_{\mu,xx}+ 10\mu B_\mu^2)w\big]\\
& = \int_\R z \mathcal L_\mu[w].
\end{align*}
Finally, it is clear that $D(\mathcal L_\mu^*)$ can be identified with $D(\mathcal L_\mu)=H^4(\R)$.
\end{proof}

A consequence of the previous result  is the fact that the spectrum of $\mathcal L_\mu$ is real-valued. Furthermore, 
the following Lemma describes the continuous spectrum of $\mathcal L_\mu$.

\begin{lem} Let $\al,\bt, \mu$ as in definition \eqref{GEBre}. 
The operator $\mathcal L_\mu$ is a compact perturbation of the constant coefficients operator
\[
\mathcal L_{0} [z]:= z_{(4x)} -2(\bt^2 -\al^2) z_{xx} +(\al^2 +\bt^2)^2 z. 
\]
In particular, the continuous spectrum of $\mathcal L_\mu$ is the closed interval $[(\al^2 +\bt^2)^2,+\infty)$ in the case $\beta\geq \al$, 
and $[ 4\al^2 \bt^2 ,+\infty)$ in the case $\beta< \al$. No embedded eigenvalues are contained in this region. The eigenvalue zero is isolated.
\end{lem}

\begin{proof}
This result is a consequence of the Weyl Theorem on continuous spectrum.  Let us note that  the nonexistence of embedded eigenvalues  
is consequence of the rapidly decreasing character of the potentials involved in the definition of $\mathcal L_\mu$.  

\medskip

The isolatedness of the zero eigenvalue is a direct consequence of standard elliptic estimates
for the eigenvalue problem associated to $\mathcal L_\mu$, corresponding uniform convergence on compact subsets of $\R$, and the 
non degeneracy of the kernel associated to $\mathcal L_\mu$. 

\medskip

\end{proof}

\medskip
Now, remembering the definition \eqref{B12} of the two directions associated to spatial translations $B_1, B_2$,
it is easy to see the following:

\begin{lem}\label{B1B2} For each $t\in \R$, one has
\[
\ker \mathcal L_\mu =\spawn \big\{ B_1(t;x_1,x_2), B_2(t;x_1,x_2)\big\}.
\]
\end{lem}

\begin{proof}
From Theorem \ref{GBGE}, one has that $\partial_{x_1}J_\mu[B_\mu] =\partial_{x_2}J_\mu[B_\mu] \equiv 0$. Writing down these identities, we obtain
\be\label{LB12}
\mathcal L_\mu [B_1](t;x_1,x_2) = \mathcal L_\mu [B_{2}](t;x_1,x_2) =0,
\ee
with $\mathcal L_\mu$ the linearized operator defined in \eqref{LGE} and $B_1, B_2$ defined in \eqref{B12}. A direct analysis involving ordinary differential 
equations shows that the null space of $\mathcal L_{0}$ is spawned by functions of the type
\[
e^{\pm \bt x } \cos(\al x), \quad e^{\pm \bt x } \sin(\al x),Ê\quad \al,\bt>0,
\]
(note that this set is linearly independent). Among these four functions, there are only two $L^2$-integrable ones in 
the semi-infinite line $[0,+\infty)$. Therefore, the null space of $\mathcal L_\mu|_{H^4(\R)}$ is spanned by at most two $L^2$-functions. Finally, 
comparing with (\ref{LB12}), we have the desired conclusion. 
\end{proof}

\medskip

We consider now the natural modes associated to the scaling parameters, which are the best candidates to generate negative directions 
for the related quadratic form defined by $\mathcal L_\mu$. Recall the definitions of $\Lambda_\al B_{\mu} $ and $\Lambda_\beta B_{\mu} $ 
introduced in (\ref{LAB}). For these two directions, one has the following
\begin{lem}\label{Scaling} 
Let $B_\mu =B_{\al,\bt,\mu}$ be any Gardner breather. Consider the scaling directions $\Lambda_\al B_\mu$ and $\Lambda_\bt B_\mu$ introduced in \eqref{LAB}. 
Then, given $\al,\bt>0$ and $\forall\mu\in(0,\mu_{max})$, we have

\be\label{Pos}
\int_\R  \Lambda_\al B_\mu \, \mathcal L_\mu [\Lambda_\al B_\mu]  = 
32 \al^2\bt\Big[1+\frac{2\mu^2\Delta}{\Delta^2+8\mu^2\bt^2}\Big]>0,
\ee
and
\medskip
\be\label{Neg}
\int_\R  \Lambda_\bt B_\mu\, \mathcal L_\mu [\Lambda_\bt B_\mu]  =  
-16 \bt\Big[(\al^2-\bt^2) + (\al^2+\bt^2 +2\mu^2)\Big(\frac{\Delta^2 +2\mu^2\Delta +4\mu^2\bt^2}{\Delta^2+8\mu^2\bt^2}
\Big)\Big]<0.
\ee

\end{lem}

\begin{proof}
 From \eqref{EcBGEfinal}, we get after derivation with respect to $\al$ and $\beta$,
\[
\mathcal L_\mu [\Lambda_\al B_\mu ] =  -4\al [ B_{\mu,xx} + B_\mu^3 + 3\mu B^2_\mu+(\al^2+\bt^2)B_\mu],\]
\[
\mathcal L_\mu [\Lambda_\bt B_\mu]  =  4\bt [ B_{\mu,xx} + B_\mu^3 + 3\mu B^2_\mu- (\al^2+\bt^2)B_\mu].
\]
We deal with the first identity \eqref{Pos}. Note that from (\ref{LAB1a}), (\ref{E2}) and (\ref{LAB3}),

\[
\begin{aligned}
&\int_\R  \Lambda_\al B_\mu\, \mathcal L_\mu [\Lambda_\al B_\mu] =  -4\al  \int_\R [  B_{\mu,xx} + B_\mu^3 + 3\mu B^2_\mu+(\al^2+\bt^2)B_\mu] \Lambda_\al B_\mu \nonu \\
&=  4\al \partial_\al E_\mu[B_\mu] - 4\al(\al^2+\bt^2)\partial_\al M[B_\mu]\nonu\\
&=4\al\cdot8\al\bt(1-\frac{4\mu^4}{\Delta^2+8\mu^2\bt^2})  - 4\al(\al^2+\bt^2)(-\frac{16\mu^2\bt\al}{\Delta^2+8\mu^2\bt^2})\nonu\\
&=32 \al^2\bt\Big[1-\frac{4\mu^4}{\Delta^2+8\mu^2\bt^2} + \frac{2\mu^2(\al^2+\bt^2)}{\Delta^2+8\mu^2\bt^2}\Big]=\eqref{Pos}.
\end{aligned}
\]

\medskip
\noindent
Following a similar analysis, we have
\begin{align*}
&\int_\R  \Lambda_\bt B_\mu \, \mathcal L_\mu [\Lambda_\bt B_\mu]  =  4\bt  \int_\R [ B_{\mu,xx} + B_\mu^3 + 3\mu B^2_\mu+(\al^2+\bt^2)B_\mu] \Lambda_\bt B_\mu  \\
&= -4\bt \partial_\bt E_\mu[B_\mu] - 4\bt (\al^2 +\bt^2)\partial_\bt M[B_\mu]\nonu \\
&= -4\bt[4(\al^2 -\bt^2) + 8\mu^2(1 + \frac{2\mu^2(\Delta - 2\bt^2)}{\Delta^2+8\mu^2\bt^2})] - 4\bt(\al^2 +\bt^2)4(1+\frac{2\mu^2(\Delta 
- 2\bt^2)}{\Delta^2+8\mu^2\bt^2})\nonu\\
&= -4\bt[4(\al^2-\bt^2) +4\frac{(\al^2+\bt^2+2\mu^2)(\Delta^2+2\mu^2\Delta+4\mu^2\bt^2)}{\Delta^2+8\mu^2\bt^2})]=\eqref{Neg}.
\end{align*}
\end{proof}

\ms 
\noindent
A direct consequence of the previous identities  and Corollary \ref{WCcor}, is the following:  

\begin{cor}\label{B_0per}
With the notation of Lemma \ref{Scaling}  let

\be\label{B0}
B_{0,\mu} := \frac{\al\Lambda_\bt B_\mu + \bt\Lambda_\al B_\mu}{8\al\bt(\al^2+\bt^2)}.\\
\ee
\ms

Then $B_{0,\mu}$ is Schwartz, satisfies $\mathcal L_\mu[B_{0,\mu}] = - B_\mu$ and $\forall \mu\in(0,\mu_{max})$

\be\label{negB0}
\int_\R B_{0,\mu} B_\mu = 
\frac{1}{2\bt(\al^2+\bt^2)}\Big(\frac{\Delta^2+2\mu^2\Delta}{\Delta^2+8\mu^2\bt^2}\Big)>0.
\ee
\quad  


Moreover, 
\bea\label{dirNegPos1}
&\frac{1}{2}\int_\R B_{0,\mu}\mathcal L_\mu[B_{0,\mu}]  < 0.
\eea
\end{cor}
\medskip

\begin{rem}
In other words, from Definition \ref{Neg_dir}, we can see $B_{0,\mu}$ as a negative direction of $\mathcal{Q}_\mu$
 $\forall\mu\in(0,\mu_{max})$. Besides that,  $B_{0,\mu}$  is not orthogonal to the breather itself.  
Note additionally that  constants involved in (\ref{negB0}) are independent of time.

\end{rem}


\begin{proof}
Using \eqref{B0}, we are lead to understanding the sign of the function 

\be\label{HG}
\begin{aligned}
\int_\R B_{0,\mu} B_\mu & =  \frac{1}{8\al\bt(\al^2+\bt^2)} \int_\R (\al\Lambda_\bt B_\mu + \bt\Lambda_\al B_\mu) B_\mu \\
& = \frac{1}{8\al\bt(\al^2+\bt^2)}(\al\partial_\bt M[B_\mu] + \bt\partial_\al M[B_\mu])\\
& =\frac{1}{8\al\bt(\al^2+\bt^2)}\Big(4\al[1+\frac{2\mu^2(\Delta-2\bt^2)}{\Delta^2+8\mu^2\bt^2}]  -  
\frac{16\mu^2\al\bt^2}{\Delta^2+8\mu^2\bt^2}\Big)=\eqref{negB0},
\end{aligned}
\ee
where $\partial_\bt M[B_\mu]$ was computed in \eqref{LAB2b}. 


\end{proof}





Now, in order to prove that $\mathcal L_\mu$ possesses, 
for all time, \emph{only one negative eigenvalue}, we follow the Greenberg 
and Maddocks-Sachs strategy \cite{Gr,MS}, applied this time to the linear, \emph{oscillatory} operator $\mathcal L_\mu$. 
More specifically, we will use the following

\begin{lem}[Uniqueness criterium, see also \cite{Gr,MS}]\label{Wr1} Let $B=B_{\mu}$ be any Gardner breather, and $B_1,B_2$ 
the corresponding kernel of the operator $\mathcal L_\mu$. Then $\mathcal L_\mu$ has 
\[
\sum_{x\in \R} \dim \ker W[B_1, B_2] (t;x)
\] 
negative eigenvalues, counting multiplicity. Here, $W$ is the Wronskian matrix of the functions $B_1$ and $B_2$, 
\be\label{WM}
W[B_1, B_2] (t;x) := \left[ \begin{array}{cc} B_1 & B_2 \\  (B_1)_x & (B_2)_x  \end{array} \right] (t,x).
\ee
\end{lem}

\begin{proof}
This result is essentially contained in \cite[Theorem 2.2]{Gr}, where the finite interval case was considered. 
As shown in several articles (see e.g. \cite{MS,HPZ}), the extension to the real line is direct and does not require 
additional efforts. We skip the details.
\end{proof}

In what follows, we compute the Wronskian (\ref{WM}). 
The surprising fact is the following greatly simplified new expression for the determinant of (\ref{WM}) and which generalizes the Wronskian for
the mKdV's breather case ($\mu=0$) (see \cite[Lemma 4.7]{AM}):

\begin{lem}\label{WMLe}
Let $B_\mu=B_{\al,\bt,\mu}$ be any Gardner breather, $B_1, B_2$ the corresponding kernel elements defined in \eqref{B12} and
$D_\mu = F_\mu^2 + G_\mu^2$.  Then

\be\label{Wsimpl}
\begin{aligned}
&\det W[B_1,B_2](t;x):=\frac{4\bt^3(\al^2+\bt^2)^2((\al^2+\bt^2)^2-4\mu^2(\al^2-\mu^2))}{\Delta^3 D_\mu^2}\\
&\Big[
\sinh(2\bt y_2) + \frac{4\bt^2\mu^2\cosh(2\bt y_2)}{(\al^2+\bt^2)^2-4\mu^2(\al^2-\mu^2)}
- \frac{\bt\Delta((\al^2+\bt^2)^2 - 2\mu^2(\al^2-\bt^2))\sin(2\al y_1)}{\al(\al^2+\bt^2)((\al^2+\bt^2)^2 - 4\mu^2(\al^2-\mu^2))}\\
&+\frac{4\bt^2\mu^2\Delta\cos(2\al y_1)}{(\al^2+\bt^2)((\al^2+\bt^2)^2 - 4\mu^2(\al^2-\mu^2))}\Big].\\
\end{aligned}
\ee
\end{lem}
\ms

\begin{proof}
We start with a very useful simplification. We claim that
\be\label{Wro}
\det W[B_1,B_2](x)  = -2(\al^2 +\bt^2)\Big[-(\mu + B_\mu)\tilde{B}_{11} + \partial^2_{x_1}\partial_x\log\Big(G^2_\mu + F^2_\mu\Big)\Big],
\ee
with $\tilde B_\mu=\tilde B_\mu(t,x; x_1,x_2)$ defined in (\ref{tBga}), and $\tilde B_j,~\tilde B_{ij},~~i,j=1,2,$ as in \eqref{Idwrons}. 
In order to prove the above simplification, we start from (\ref{2ndga}), and taking derivative with respect to $x_1$ and $x_2$, we get
\be\label{eqB1}
(B_1)_{xx} + (\tilde B_1)_t + 3B_\mu^2 B_1 + 6\mu B_\mu B_1 =0, \quad (B_2)_{xx} + (\tilde B_2)_t + 3B_\mu^2 B_2 + 6\mu B_\mu B_2=0.
\ee
Multiplying the first equation above by $B_2$ and the second by $-B_1$, and adding both equations, we obtain
\[
(B_1)_{xx}B_2 - (B_2)_{xx} B_1 + (\tilde B_1)_t B_2 -(\tilde B_2)_t B_1 =0,
\]
that is, 
\be\label{interm}
( (B_1)_x B_2 - (B_2)_x B_1)_x =  (\tilde B_2)_t B_1-(\tilde B_1)_t B_2.
\ee
On the other hand, since we are working with smooth functions, one has $B_\mu =\tilde B_1 + \tilde B_2$,
\[
B_1 = \tilde B_{11} + \tilde B_{12}, \quad B_2 = \tilde B_{12} + \tilde B_{22},
\]
and
\[
(\tilde B_1)_t = \delta  \tilde B_{11} +  \ga \tilde B_{12}, \quad (\tilde B_2)_t =  \delta\tilde B_{12} + \ga \tilde B_{22}.
\]
Substituting into (\ref{interm}), we get
\[
( (B_1)_x B_2 - (B_2)_x B_1)_x = (\delta- \ga) (\tilde B_{12}^2 -\tilde B_{11}\tilde B_{22}).
\]
\noindent
Now, integrating in $x$ and using the nonlinear identity \eqref{Idwrons} we get

\[
 \det W[B_1,B_2](x)  = -2(\al^2 +\bt^2)\Big[-(\mu + B_\mu)\tilde{B}_{11} + \partial^2_{x_1}\partial_x\log\Big(G^2_\mu + F^2_\mu\Big)\Big].
\]

\medskip
\noindent
Finally to prove \eqref{Wsimpl}, we write explicitly the two terms involved at the r.h.s. of equation above. We  will follow notation
of Appendix \ref{apIddificil}, but this time changing $\partial_t$ by $\partial_{x_1}$, and therefore defining 
$G:=G_\mu,~G_1:=G_x,~G_2:=G_{x_1},~G_3:=G_{xx_{1}},~G_4:=G_{x_{1}x_{1}},~G_5:=G_{xx_{1}x_{1}}$ and 
$F:=F_\mu,~F_1:=F_x,~F_2:=F_{x_1},~F_3:=F_{xx_{1}},~F_4:=F_{x_{1}x_{1}},~F_5:=F_{xx_{1}x_{1}}$. Hence we get,

\[
 -(\mu + B_\mu)=-\frac{\mu D_\mu -2\sqrt{2}(-F_1G+FG_1)}{D_\mu},
\]

\[
 \tilde{B}_{11}=\frac{2\sqrt{2}}{D_\mu^2}[-G^2(F_4G-2F_2G_2)-F^2(F_4G+2F_2G_2) + F^3G_4 + FG(2F^2_2 - 2G^2_2 + GG_4)],
\]
and finally
\be\label{1term}
-(\mu + B_\mu)\tilde{B}_{11}=-2\sqrt{2}\frac{M_1}{D_\mu^3},
\ee
\noindent
with
\be\label{M1term}\begin{aligned}
M_1:=\Big[\mu D_\mu -2\sqrt{2}(-F_1G+FG_1)\Big]&\Big[-G^2(F_4G-2F_2G_2)-F^2(F_4G+2F_2G_2)\\
&+ F^3G_4 + FG(2F^2_2 - 2G^2_2 + GG_4)\Big].
 \end{aligned}
\ee
Similarly, for the second term in \eqref{Wro} at the r.h.s. we have

\be\label{2term}\begin{aligned}
 &\partial^2_{x_1}\partial_x\log\Big(G^2_\mu + F^2_\mu\Big)=\frac{M_2}{D_\mu^3},
 \end{aligned}
\ee
with

\be\label{M2term}
\begin{aligned}
M_2:=&16(FF_1+GG_1)(FF_2+GG_2)^2-4D_\mu(FF_1+GG_1)(F_2^2+FF_4+G_2^2+GG_4)\\
&-8D_\mu^2(FF_2+GG_2)(F_1F_2+FF_3+G_1G_2+GG_3)\\
&+2D_\mu^2(2F_2F_3+F_1F_4+FF_5+2G_2G_3+G_1G_4+GG_5).\\
\end{aligned}
\ee

\ms
\noindent
Therefore putting together \eqref{1term} and \eqref{2term}, we get

\be\label{1paso}
\begin{aligned}
&-(\mu + B_\mu)\tilde{B}_{11} +\partial^2_{x_1}\partial_x\log\Big(G^2_\mu + F^2_\mu\Big) = \frac{-2\sqrt{2}M_1 + M_2}{D_\mu^3}.
\end{aligned}\ee

\noindent
Indeed, it is possible to see that the above numerator reduces to
\be\label{2paso}
 -2\sqrt{2}M_1 + M_2 = 2D_\mu M_3,
\ee
where $M_3$ is given by

\be\label{M3}\begin{aligned}
&M_3:=\\
&F^2[-2F_2F_3-F_1F_4+2G_2G_3-3G_1G_4+GG_5+\sqrt{2}\mu GF_4+2\sqrt{2}\mu F_2G_2]+ F^3(F_5-\sqrt{2}\mu G_4) \\
&+F[F_5G^2 + 2F_4GG_1 - 4F_3GG_2 + 4F_2G_1G_2 - 4F_2GG_3 - \sqrt{2}\mu(2F_2^2G-2GG_2^2+G^2G_4)\\
& + 2F_1(F_2^2-G_2^2+GG_4)] +G[-3F_1F_4G-2F_2^2G_1+2G_1G_2^2-2GG_2G_3-GG_1G_4+G^2G_5\\
& +\sqrt{2}\mu F_4G^2 + 2F_2(F_3G + 2F_1G_2 -\sqrt{2}\mu GG_2)].\\
\end{aligned}\ee
\noindent
We verify, using the symbolic software \emph{Mathematica}, that after substituting  $F's$ and $G's$ terms explicitly 
in $M_3$  and lengthy rearrangements, \eqref{M3}  simplifies as follows:

\be\label{M3simplify}\begin{aligned}
&2M_3 = \frac{-2\bt^3(\al^2+\bt^2)((\al^2+\bt^2)^2-4\mu^2(\al^2-\mu^2))}{\Delta_\mu^3}\\
&\Big[
\sinh(2\bt y_2) + \frac{4\bt^2\mu^2\cosh(2\bt y_2)}{(\al^2+\bt^2)^2-4\mu^2(\al^2-\mu^2)}
- \frac{\bt\Delta_\mu((\al^2+\bt^2)^2 - 2\mu^2(\al^2-\bt^2))\sin(2\al y_1)}{\al(\al^2+\bt^2)((\al^2+\bt^2)^2 - 4\mu^2(\al^2-\mu^2))}\\
&+\frac{4\bt^2\mu^2\Delta_\mu\cos(2\al y_1)}{(\al^2+\bt^2)((\al^2+\bt^2)^2 - 4\mu^2(\al^2-\mu^2))}\Big].\\
\end{aligned}\ee

\ms
\noindent
Finally we get
\be\label{3paso}
\begin{aligned}
-(\mu + B_\mu)\tilde{B}_{11} +\partial^2_{x_1}\partial_x\log\Big(G^2_\mu + F^2_\mu\Big) &= \frac{-2\sqrt{2}M_1 + M_2}{D_\mu^3} = \frac{2D_\mu M_3}{D_\mu^3}
=\frac{\eqref{Wsimpl}}{-2(\al^2 +\bt^2)}.
\end{aligned}\ee

\end{proof}

\begin{prop}\label{WW} The operator $\mathcal L_\mu$  defined in \eqref{LGE} and for every $\mu\in(0,\mu_{max})$ has a 
unique negative eigenvalue $-\la_0^2<0$, of multiplicity one, and  $\la_0=\la_0(\al,\bt,\mu,x_1,x_2,t)$.
\end{prop}

\begin{proof}
We compute the determinant (\ref{WM}) required by Lemma \ref{Wr1}. From Lemma \ref{WMLe},  after a standard translation argument, 
we will denote $\tilde{y}_2 = y_2 + (\delta -\ga)t  + \tilde x_2$, and we just need to consider the behavior of the function
\be\label{fy2}
\begin{aligned}
&f_\mu(y_2) =f_{t,\al,\bt,\mu,\tilde x_2}(y_2) :=  \sinh(2\bt y_2) + \frac{4\bt^2\mu^2\cosh(2\bt y_2)}{(\al^2+\bt^2)^2-4\mu^2(\al^2-\mu^2)}\\
&- \frac{\bt\Delta((\al^2+\bt^2)^2 - 2\mu^2(\al^2-\bt^2))\sin(2\al\tilde{y}_2 )}{\al(\al^2+\bt^2)((\al^2+\bt^2)^2 
- 4\mu^2(\al^2-\mu^2))}
+\frac{4\bt^2\mu^2\Delta\cos(2\al \tilde{y}_2)}{(\al^2+\bt^2)((\al^2+\bt^2)^2 - 4\mu^2(\al^2-\mu^2))}.
\end{aligned}
\ee
for $\tilde x_2 := x_1 -x_2 \in \R$, and $\delta-\ga = -2(\al^2 +\bt^2).$

\medskip

A simple argument shows that for $y_2\in \R$ such that 

$$|\sinh (2\bt y_2)| > \frac{\bt\Delta((\al^2+\bt^2)^2 - 2\mu^2(\al^2-\bt^2))}{\al(\al^2+\bt^2)((\al^2+\bt^2)^2 
- 4\mu^2(\al^2-\mu^2))}  +\frac{4\bt^2\mu^2\Delta}{(\al^2+\bt^2)((\al^2+\bt^2)^2 - 4\mu^2(\al^2-\mu^2))},$$ 

\medskip
\noindent
$f_\mu$ has no root. Moreover, there exists $R_0=R_0(\al,\bt,\mu)>0$ such that, for all $y_2>R_0$ one has $f_\mu(y_2) >0$ and for all $y_2<-R_0$, $f_\mu(y_2)<0.$ 
Therefore, since $f_\mu$ is continuous, there is a root $y_0=y_0(t,\al,\bt,\mu,\tilde x_2)\in [-R_0, R_0]$ for $f_\mu$. 
Additionally,  we have that
\[
\begin{aligned}
&f_\mu^{'}(y_2) = 2\bt [ \cosh(2\bt y_2) + \frac{4\bt^2\mu^2 \sinh(2\bt y_2)}{(\al^2+\bt^2)^2-4\mu^2(\al^2-\mu^2)}\\
&-\frac{\Delta((\al^2+\bt^2)^2-2\mu^2(\al^2-\bt^2))\cos (2\al \tilde{y}_2)}{(\al^2+\bt^2)((\al^2+\bt^2)^2-4\mu^2(\al^2-\mu^2))}
-\frac{4\bt\al\mu^2\Delta\sin(2\al\tilde{y}_2)}{(\al^2+\bt^2)((\al^2+\bt^2)^2-4\mu^2(\al^2-\mu^2))}],
\end{aligned}
\]
\noindent
hence, two cases have to be considered. We remember here that $\mu\in(0,\mu_{max})$. In the first case, when $\mu\rightarrow0^+\equiv\epsilon,$ we have
that $f_\mu^{'}(y_2)\approx\cosh(2\bt y_2) - \cos(2\al (y_2 - 2(\al^2+\bt^2)t + \tilde x_2)) + O(\epsilon)>0$ if $y_2\neq 0$. On the other side, when 
$\mu^2\rightarrow\mu_{max}^{2-}\equiv\mu_{max}^2-\frac{\epsilon}{2}$, with
$\epsilon\lll1$, we have that 

\[
\begin{aligned}
f_\mu^{'}(y_2) &= 2\bt [ \cosh(2\bt y_2) + \sinh(2\bt y_2) - \frac{\epsilon^2\sinh(2\bt y_2)}{2\bt^2(\al^2+\bt^2-\epsilon)+\epsilon^2}\\
&-\frac{\epsilon(2\bt^4-\epsilon\bt^2+\al^2(2\bt^2+\epsilon))\cos(2\al\tilde{y}_2)}{(\al^2+\bt^2)(2\al^2\bt^2+2\bt^4-2\epsilon\bt^2+\epsilon^2)}
-\frac{2\al\bt\epsilon(\al^2+\bt^2-\epsilon)\sin(2\al\tilde{y}_2)}{(\al^2+\bt^2)(2\al^2\bt^2+2\bt^4-2\epsilon\bt^2+\epsilon^2)}]>0,\\\\
\end{aligned}
\]
\ms
\noindent
by simple inspection. Therefore, if $y_0\neq 0$ then it is unique and then
\[
\sum_{x\in \R} \dim \ker W[B_1, B_2] (t; x) = \dim \ker W[B_1,B_2](t; y_0 - \ga t  -x_2) =1,
\]
since $B_1$ or $(B_1)_x$ are not zero at that time. Indeed, it is enough to show that $W[B_1,B_2](t,x)$ is not identically zero, then $\dim \ker W[B_1,B_2]<2$. In order to prove this fact, 
note that from (\ref{eqB1}) $B_1$ solves, for $t,x_1,x_2\inÊ\R $ fixed, a second order linear ODE with source term $-(\tilde B_\mu)_t$. 
Therefore, by standard well-posedness results, both $B_1$ and $(B_1)_x$ cannot be identically zero at the same point.

\end{proof}

\medskip

We consider now some standard remarks. We can reduce the spectral problem to another independent of time. Indeed, 
from \eqref{dege} and after  translation and redefinition of the shift parameters $x_1$ and $x_2$ we can assume that 
\[
B_\mu=B_{\al,\bt,\mu}(0,x; x_1,0), \quad x_1\in [0, \frac{2\pi}{\al}].
\]
In what follows we assume that $B_\mu$ is given by the previous formula.


\begin{rem}\label{remCoer}
Let $z\in H^2(\R)$, and $B_\mu=B_{\al,\bt,\mu}$ be any Gardner breather. Now remembering the quadratic form \eqref{Qmu} associated to $\mathcal L_\mu$,
$\mathcal Q_\mu[z]  = \int_\R z \mathcal L_\mu[z]$ and from Lemma \ref{B1B2}, it is easy to see that
$\mathcal Q_\mu[B_1] =\mathcal Q_\mu[B_2]=0$. Moreover, 
inequality \eqref{Pos} means that $\Lambda_\al B_\mu$ is a positive direction for $\mathcal Q_\mu$ when $\mu\in(0,\mu_{max})$. 
Additionally,  from (\ref{Qmu}) $\mathcal Q_\mu$ is bounded below, namely
\[
\mathcal Q_\mu[z] \geq -c_{\al,\bt,\mu}\|z\|_{H^2(\R)}^2,
\]

\end{rem}

Let $ B_{-1} \in \mathcal S\backslash \{ 0\}$ be an eigenfunction associated to the unique negative eigenvalue of the operator $\mathcal L_\mu$, as stated in Proposition \ref{WW}. 
We assume that $ B_{-1} $ has unit $L^2$-norm, so $B_{-1}$ is now unique. In particular, one has $\mathcal L_\mu [ B_{-1}] =-\la_0^2 B_{-1} .$ It is clear from Proposition
\ref{WW} and Lemma \ref{B1B2} that the following result holds.

\begin{lem}\label{Coerci}
There exists a continuous function $\nu_0 =\nu_0(\al,\bt,\mu)$, 
well-defined and positive for all $\al,\bt>0,$ with $\mu\in(0,\mu_{max}),$ and such that, for all $z_0\in H^2(\R)$ satisfying
\be\label{Or1}
\int_\R z_0 B_{-1} =\int_\R z_0 B_1 =\int_\R z_0 B_2 =0,
\ee
then 
\be\label{coee}
\mathcal Q_\mu[ z_0] \geq \nu_0\| z_0\|_{H^2(\R)}^2.
\ee
\end{lem}

\begin{proof}

The existence of a \emph{positive} constant $\nu_0=\nu_0(\al,\bt,\mu,x_1)$ such that (\ref{coee}) is satisfied is now clear
from Remark \ref{remCoer} and the three orthogonality conditions. Moreover, thanks to the periodic character of the  
variable $x_1$, and the nondegeneracy of the kernel,  we obtain a uniform, positive bound independent of 
$x_1,x_2$ and $t$, still denoted  $\nu_0$. The proof is complete.
\end{proof}

Since $B_{-1}$ is difficult to work with, we look for an  easier version of the previous result. 
We can easily prove, as in \cite[Proposition (4.11)]{AM}, 
that the eigenfunction $B_{-1}$ associated to the negative eigenvalue of $\mathcal L_\mu$ can be replaced by the breather itself, 
which has better behavior in terms of error controlling, unlike the first eigenfunction. This simple fact allows us to prove the nonlinear stability
result as in the standard approach, without using scaling modulations.
Recall that using the first eigenfunction as orthogonality condition does not guarantee a suitable control on the scaling modulation
parameter, because the control given by this direction might be not good enough to close the stability estimates. However, 
the breather can be used as an alternative direction, and all these previous arguments remain valid, exactly as in \cite{AM}, 
provided the Weinstein's sign condition 
\be\label{Cond3}
\int_\R B_{0,\mu} B_\mu >0  \qquad \hbox{\Big(or equivalently $\displaystyle{\int_\R B_{0,\mu} \mathcal L_\mu[B_{0,\mu}] <0}$\Big),} 
\ee

\ms
\noindent
do hold. Note that this precisely is what happens in \eqref{dirNegPos1} when $\mu\in(0,\mu_{max})$.
\ms

\begin{prop}\label{PropOrtog} Let $B_\mu=B_{\al,\bt,\mu}$ be any Gardner breather, and $B_1,B_2$ the corresponding kernel
of the associated operator $\mathcal L_\mu$. Let $\al,\bt>0$ and $\mu\in(0,\mu_{max})$.
There exists $\sigma_0>0$, depending on $\al,\bt,\mu$ only, such that, for any  $ z\in H^2(\R)$ satisfying
\be\label{OrthoK}
\int_\R B_1  z = \int_\R B_2 z =0,
\ee
one has
\be\label{Coer}
\mathcal Q_\mu[ z] \geq \sigma_0\| z\|_{H^2(\R)}^2 -\frac{1}{\sigma_0}\Big(\int_\R z B_\mu \Big)^2.
\ee
\end{prop}

\begin{proof}
We follow a similar strategy as stated in \cite{AM}, and we include here for the sake of completeness.  
Indeed, it is enough to prove that, under hypothesis $\mu\in(0,\mu_{max}),$ the conditions (\ref{OrthoK}) and the additional orthogonality condition $\int_\R zB_\mu =0$, one has
\[
\mathcal Q_\mu[ z] \geq \sigma_0\| z\|_{H^2(\R)}^2.
\]
In what follows we prove that we can replace $ B_{-1}$ by the breather $B_\mu$ in Lemma \ref{Coerci} and the result 
essentially does not change. Indeed, note that from (\ref{B0}), the function $B_{0,\mu}$ satisfies $\mathcal L_\mu[B_{0,\mu}] = - B_\mu$, 
and from (\ref{negB0}) and hypothesis $\mu\in(0,\mu_{max})$,
\be\label{posB0}
\int_\R B_{0,\mu} B_\mu = -\int_\R B_{0,\mu} \mathcal L_\mu[B_{0,\mu}] =-\mathcal Q_\mu[B_{0,\mu}] > 0.
\ee
The next step is to decompose $z$ and $B_{0,\mu}$ in $\spawn (B_{-1}, B_1,B_2)$ and the corresponding orthogonal subspace. 
Using the same notation that in \cite[Prop 4.13]{AM}, one has
\[
z =\tilde z +  m B_{-1}, \quad  B_{0,\mu}=  b_0 +  n  B_{-1} + p_1 B_1 + p_2 B_2, \quad m,n, p_1,p_2\in \R,
\]
where 
\[
\int_\R \tilde z  B_{-1} =\int_\R \tilde z B_1=\int_\R \tilde z B_2 =\int_\R   b_0 B_{-1} =\int_\R  b_0 B_1=\int_\R  b_0 B_2 =0. 
\]
Note in addition that 
\[
\int_\R B_{-1}B_1 =\int_\R B_{-1}B_2 =0.
\]
From here and the previous identities we have
\be\label{Qzz}
\mathcal Q_\mu[z] =\int_\R (\mathcal L_\mu [\tilde z] - m\la_0^2 B_{-1})(\tilde z +m B_{-1}) = \mathcal Q_\mu[\tilde z] -m^2 \la_0^2. 
\ee
Now, since $\mathcal L_\mu[B_{0,\mu}] =-B_\mu$, one has 
\begin{align}
0 & = \int_\R z B_\mu = -\int_\R z\mathcal L_\mu[B_{0,\mu}]  =\int_\R \mathcal L_\mu[\tilde z +m B_{-1}] B_{0,\mu}\nonu\\
&  = \int_\R (\mathcal L_\mu[\tilde z] -m\la_0^2 B_{-1})( b_0 +n B_{-1} + p_1 B_1 + p_2 B_2) \ =  \int_\R \mathcal L_\mu[\tilde z]  b_0 -mn\la_0^2.\label{zB}
\end{align}
On the other hand,  from Corollary \ref{B_0per},
\be\label{B0B}
\int_\R B_{0,\mu} B_\mu =  -\int_\R B_{0,\mu} \mathcal L_\mu[B_{0,\mu}] = -\int_\R ( b_0 + n B_{-1}) (\mathcal L_\mu[b_0] -n\la_0^2 B_{-1})
= -\mathcal Q_\mu[b_0] + n^2 \la_0^2. 
\ee
Substituting (\ref{zB}) and (\ref{B0B}) into (\ref{Qzz}), we get
\be\label{deco}
\mathcal Q_\mu[z] = \mathcal Q_\mu[\tilde z] -\frac{\displaystyle{\Big(\int_\R \mathcal L_\mu[\tilde z] b_0\Big)^2}}{\displaystyle{\int_\R B_{0,\mu}B_\mu 
+ \mathcal Q_\mu[b_0]}}.  
\ee
Note that from (\ref{posB0}) and (\ref{coee}) both quantities in the denominator are positive. 
Additionally, note that if $\tilde z =\la b_0$, with $\la\neq 0$, then
\[
\Big(\int_\R \mathcal L_\mu[\tilde z] b_0 \Big)^2 = \mathcal Q_\mu[\tilde z] \mathcal Q_\mu[b_0].
\]
In particular, if $\tilde z =\la b_0$,
\be\label{condit}
\frac{\displaystyle{\Big(\int_\R \mathcal L_\mu[\tilde z] b_0\Big)^2}}{\displaystyle{\int_\R B_{0,\mu}B_\mu + \mathcal Q_\mu[b_0]}} \leq a\,  
Q_\mu[\tilde z], \quad 0<a<1.
\ee
In the general case,  we get the same conclusion as before. Namely, choosing
\[
z =\tilde z +  m B_{-1} + q_1 B_1 + q_2 B_2, \quad m, q_1,q_2\in \R,
\]
\noindent
and  using the orthogonal decomposition induced by the scalar product 
$(\mathcal L_\mu \cdot ,Ê\cdot)_{L^2} $ on $\spawn (B_{-1}, B_1,B_2)$, we get

\[\begin{aligned}   
&\mathcal Q_\mu[z] = \int_\R z\mathcal L_\mu[z] = \int_\R (\tilde z +  m B_{-1} + q_1 B_1 + q_2 B_2)(\mathcal L_\mu [\tilde z] - m\la_0^2 B_{-1})\nonu\\
&=  \mathcal Q_\mu[\tilde z] -m^2 \la_0^2,\end{aligned}\]
\noindent
and following the same steps as above, we conclude. Therefore, we have proved (\ref{condit}) for all possible $\tilde z$. 
Finally, substituting into (\ref{deco}) and (\ref{Qzz}), $\mathcal Q_\mu[z] \geq (1-a) \mathcal Q_\mu[\tilde z] \geq 0$,  and 
$\mathcal Q_\mu[\tilde z] \geq m^2 \la_0^2$. We have, for some $C>0$,
\begin{align*}
\mathcal Q_\mu[z] & \geq  (1-a)\mathcal Q_\mu[\tilde z]  \geq \frac 12 (1-a)\mathcal Q_\mu[\tilde z] + (1-a)m^2 \la_0^2 \\
&\geq  \frac 1C(2 \|\tilde z\|_{H^2(\R)}^2 +  2m^2 \|B_{-1}\|_{H^2(\R)}^2) \geq \frac 1C\|z\|_{H^2(\R)}^2.
\end{align*}
\end{proof}
\medskip

\section{Proof of the Main Theorem}\label{6}

\medskip

In this section we prove a detailed version of  Theorem \ref{T1p8}.

\begin{thm}[$H^2$-stability of Gardner breathers]\label{T1gardner} Let $\al, \bt \in \R\backslash\{0\}$ and 
$\mu\in(0,\mu_{\max})$.
There exist positive parameters $\eta_0, A_0$, depending on $\al,\beta$ and $\mu$, such that the following holds.  Consider $w_0 \in H^2(\R)$,
and assume that there exists $\eta \in (0,\eta_0)$ such that 
\be\label{In}
\|  w_0 - B_{\mu}(t=0;0,0) \|_{H^2(\R)} \leq \eta.
\ee
Then there exist $x_1(t), x_2(t)\in \R$ such that the solution $w(t)$ of the Cauchy problem for the Gardner equation 
\eqref{GE}, with initial data $w_0$, satisfies
\be\label{Fn1}
\sup_{t\in \R}\big\| w(t) - B_{\mu}(t; x_1(t),x_2(t)) \big\|_{H^2(\R)}\leq A_0 \eta,
\ee
with
\be\label{Fn2}
\sup_{t\in \R}|x_1'(t)| +|x_2'(t)| \leq KA_0 \eta,
\ee
for a constant $K>0$.
\end{thm}

\begin{rem}
The initial condition (\ref{In}) can be replaced by any initial breather profile of the form 
$\hat{B}_\mu := B_{\al,\bt,\mu}(t_0; x_1^0, x_2^0)$, with $t_0, x_1^0, x_2^0 \in \R$, 
thanks to the invariance of the equation under translations in time and space. 
In addition, a similar result is available for the \emph{negative} breather $-B_{\al,\bt,-|\mu|}$ which is also a solution of (\ref{GE}).
\end{rem}

\medskip

\begin{proof}[Proof of Theorem \ref{T1gardner}]
The proof of this result is completely similar to the proof of the $H^2-$stability of mKdV breathers \cite[Theorem (6.1)]{AM}, after following the same steps,
and once we guarantee  coercivity of the bilinear, form $\mathcal Q_\mu$ (see Proposition \ref{PropOrtog} and Lemma \ref{Coerci}), and the nonvanishing 
of the denominator $\int_\R B_{0,\mu}B_\mu + \mathcal Q_\mu[b_0]$ appearing in \eqref{deco}. 
For the sake of completeness, we include here a sketch with some steps.

\ms

1. From the continuity of the Gardner 
flow for $H^2(\R)$ data, there exists a time $T_0>0$ and continuous parameters $x_1(t), x_2(t)\in \R$, defined for all $t\in [0, T_0]$, 
and such that the solution $w(t)$ of the Cauchy problem for the Gardner equation (\ref{mKdV}), with initial data $w_0$, satisfies
\be\label{F0}
\sup_{t\in [0, T_0]}\big\| w(t) - B_\mu(t; x_1(t),x_2(t)) \big\|_{H^2(\R)}\leq 2 \eta.
\ee
The idea is to prove that $T_0 =+\infty$. In order to do this, let $K^*>2$ be a constant, to be fixed later. 
Let us suppose, by contradiction, that the \emph{maximal time of stability} $T^*$, namely
\begin{align}\label{Te}
T^* &:=   \sup \Big\{ T>0 \, \big| \, \hbox{ for all } t\in [0, T], \, \hbox{ there exist } \tilde x_1(t), \tilde x_2(t) \in \R  \hbox{ such that } \nonu \\
&  \qquad \qquad \sup_{t\in [0, T]}\big\| w(t) - B_\mu(t; \tilde x_1(t), \tilde x_2(t)) \big\|_{H^2(\R)}\leq K^* \eta \Big\},
\end{align}
is finite. It is clear from (\ref{F0}) that $T^*$ is a well-defined quantity. Our idea is to find a suitable contradiction to the assumption $T^*<+\infty.$

\ms

2. After that, we  apply a well known theory of modulation for the solution $w(t)$.%

\begin{lem}[Modulation]\label{Mod} There exists $\eta_0>0$ such that, for all $\eta\in (0, \eta_0)$, the following holds. 
There exist $C^1$ functions $x_1(t)$, $x_2(t) \in \R$, defined for all $t\in [0, T^*]$, and such that 
\be\label{z}
z(t) := w(t) - {B}_\mu (t), \quad B_\mu(t,x) := B_{\al,\bt,\mu} (t,x; x_1(t),x_2(t)) 
\ee
satisfies, for $t\in [0, T^*]$,
\be\label{Ortho}
\int_\R B_1(t; x_1(t),x_2(t)) z(t) = \int_\R B_2(t; x_1(t),x_2(t)) z(t)=0.
\ee
Moreover, one has
\be\label{apriori}
\|z(t)\|_{H^2(\R)} + |x_1'(t)| + |x_2'(t)| \leq K K^* \eta, \quad \|z(0)\|_{H^2(\R)} \leq K\eta,
\ee
for some constant $K>0$, independent of $K^*$. 
\end{lem}

The proof of this Lemma  is a classical application of the Implicit Function Theorem  to the function
\[
J_j(w(t), x_1,x_2) := \int_\R (w(t,x) - B_\mu (t,x; x_1,x_2)) B_j(t,x; x_1,x_2)dx, \quad j=1,2.
\]

\ms 

3. Now, we apply Lemma \ref{crit} to the function $w(t)$. Since $z(t)$ defined by (\ref{z}) is small, we get from (\ref{EE}) and Corollary \ref{Cor32}:
\be\label{Hut}
\mathcal H_\mu[w](t) = \mathcal H_\mu[B_\mu](t) + \frac 12 \mathcal Q_\mu[z](t) + N_\mu[z](t).
\ee 
Note that $|N_\mu[z](t)|\leq K \|z(t)\|_{H^1(\R)}^3$. On the other hand, by the translation invariance in space,
\[
 \mathcal H_\mu[B_\mu](t) =\mathcal H_\mu[B_\mu](0) =\hbox{constant}.
\]
Indeed, from \eqref{GEBre}, and for  $t_0(t)= \frac{x_1(t)-x_2(t)}{2(\al^2+\bt^2)}, \quad x_0(t) =  \frac{\delta x_2(t)-\ga x_1(t)}{2(\al^2 + \bt^2)},$ we have 
\[
B_\mu(t,x; x_1(t), x_2(t)) = B_\mu(t-t_0(t), x-x_0(t)).
\]
Since $\mathcal H_\mu$ 
involves integration in space of polynomial functions on $B_\mu, B_{\mu,x}$ and $B_{\mu,xx}$, we have 
\[
 \mathcal H_\mu[B_\mu(t, \cdot ; x_1(t),x_2(t))] = \mathcal H_\mu[B_\mu(t -t_0(t), \cdot -x_0(t); 0,0)] =   \mathcal H_\mu[B_\mu(t-t_0(t), \cdot ; 0,0)].
\] 
Finally, $ \mathcal H_\mu[B_\mu(t-t_0(t), \cdot ; 0,0)] = \mathcal H_\mu[B_\mu(\cdot , \cdot ; 0,0)] (t-t_0(t))$. Taking time derivative,
\[
 \partial_t \mathcal H_\mu[B_\mu(t, \cdot ; x_1(t),x_2(t))] =  \mathcal H_\mu^{'}[B_\mu(\cdot , \cdot ; 0,0)] (t-t_0(t)) \times (1-t_0'(t)) \equiv 0,
\]
hence $\mathcal H_\mu[B_\mu]$ is constant in time. Now we compare (\ref{Hut}) at times $t=0$ and $t\leq T^*$. From Lemma \ref{dF10} and  (\ref{crit}) we have
\[
\mathcal Q_\mu[z](t) \leq  \mathcal Q_\mu[z](0) + K\|z(t)\|_{H^2(\R)}^3 +K\|z(0)\|_{H^2(\R)}^3 \leq K\|z(0)\|_{H^2(\R)}^2+K\|z(t)\|_{H^2(\R)}^3. 
\]
Additionally, from (\ref{OrthoK})-(\ref{Coer}) applied this time to the time-dependent function $z(t)$, which satisfies (\ref{Ortho}), we get
\begin{align}
\|z(t)\|_{H^2(\R)}^2&  \leq     K\|z(0)\|_{H^2(\R)}^2 + K\|z(t)\|_{H^2(\R)}^3 + K\abs{\int_\R B_\mu(t) z(t)}^2\nonu\\
& \leq  K\eta^2 + K(K^*)^3 \eta^3 +K\abs{\int_\R B_\mu(t) z(t)}^2. \label{Map}
\end{align}
\noindent
4. {\bf Conclusion of the proof.} Using the conservation of mass (\ref{M1}), we have, after expanding $w=B_\mu+z$,
\begin{align*}
\abs{\int_\R  B_\mu(t)z(t)} &  \leq  K\abs{\int_\R  B_\mu(0)z(0)} +K\|z(0)\|_{H^2(\R)}^2+ K\|z(t)\|_{H^2(\R)}^2 \nonu \\
& \leq    K (\eta + (K^*)^2 \eta^2), \quad \hbox{ for each $t\in [0, T^*]$.}
\end{align*}
Substituting this last identity into (\ref{Map}), we get
\[
\|z(t)\|_{H^2(\R)}^2 \leq  K \eta^2 ( 1+ (K^*)^2 \eta^3) \leq \frac 12 (K^*)^2 \eta^2 ,
\]
by taking $K^*$ large enough. This last fact contradicts the definition of $T^*$ and therefore the stability property (\ref{Fn1}) holds true. 
Finally, (\ref{Fn2}) is a consequence of (\ref{apriori}).

\bigskip\bigskip\bigskip

\end{proof}

\appendix

\section{Proof of Lemma \ref{Id_GEq}.}\label{Id_BilinearGE}

From the Gardner equation \eqref{GE}, we select an ansatz for $w$ as
\[
w(t,x):=\phi_x,~~\phi(t,x):=\sqrt{2}i\log(\frac{{\bf G}(t,x)}{{\bf F}(t,x)}),\quad \text{where}\quad 
{\bf F} := F_\mu + iG_\mu,~~~ {\bf G} = F_\mu -iG_\mu = {\bf F}^{*}.
 \]
First note  that here $F_\mu$ and $G_\mu$ are not necessarily 
the same functions introduced in \eqref{GEBre} but generic ones for this ansatz. Then, substituting the 
above expression in \eqref{GE}, and using  Hirota's bilinear operators,  we arrive to the following conditions on ${\bf G}$ and  ${\bf F}$:

\bea\label{Bil_GE}
& D_t({\bf G}{\bf F}) + D_x^3({\bf G}{\bf F}) = 0,\nonu\\
& D_x^2({\bf G}{\bf F}) -i\sqrt{2}\mu D_x({\bf G}{\bf F}) = 0.
\eea
Then, dividing by ${\bf G}{\bf F}$ the second equation in \eqref{Bil_GE}, and taking into account the following identity\footnote{See \cite[p.152]{Mat} for further reading.}:
\[
\frac{D_x^2({\bf G}{\bf F})}{{\bf G}{\bf F}}= \partial_x^2\log({\bf G}{\bf F}) + (\partial_x\log(\frac{{\bf G}}{{\bf F}}))^2,
 \]
\noindent
 we obtain:
\bea\label{Bil_GE1}
& \frac{D_x^2({\bf G}{\bf F})}{{\bf G}{\bf F}} -i\sqrt{2}\mu \frac{D_x({\bf G}{\bf F})}{{\bf G}{\bf F}} &= 
\partial_x^2\log({\bf G}{\bf F}) + (\partial_x\log(\frac{{\bf G}}{{\bf F}}))^2 
-i\sqrt{2}\mu\partial_x\log(\frac{{\bf G}}{{\bf F}}) \nonumber\\
&& = \partial_x^2\log({\bf G}{\bf F}) + (\frac{w}{i\sqrt{2}})^2 - \mu w = 0.\nonu
\eea
Hence,
\[
w^2 =  2\frac{\partial^2}{\partial x^2}\log({\bf G}\cdot{\bf F}) - 2\mu w,
\]

\ms

\noindent
and the proof is complete. Indeed, with the same steps, it can be proved a similar result for the mKdV equation with NVBC. In fact, we arrive to  the
following relations for any solution of the mKdV with NVBC $\mu$:
\bea\label{Bil_mkdvNVBC}
& D_t({\bf G}{\bf F}) + D_x^3({\bf G}{\bf F}) +3\mu^2D_x({\bf G}{\bf F}) = 0,\nonu\\
& D_x^2({\bf G}{\bf F}) -i\sqrt{2}\mu D_x({\bf G}{\bf F}) = 0.
\eea
And, with the same steps than in the Gardner case, we obtain that any mKdV solution $u$ with NVBC $\mu$ satisfies:
\[
u^2 = \mu^2 + 2\frac{\partial^2}{\partial x^2}\log({\bf G}\cdot{\bf F}).
\]

\section{ Proof of Lemma \ref{Iddificil}.}\label{apIddificil}
Firstly and  for the sake of simplicity, we will use the following notation:

\begin{align}
& A_1:= (2(\bt^2-\al^2) + 5\mu^2),\quad A_2:= ((\al^2+\bt^2)^2 + 6\mu^2(\bt^2-\al^2 + \frac{3}{2}\mu^2)),\nonu\\
&\nonu\\
&\Delta=\al^2+\bt^2-2\mu^2,\quad e^{z} = \cosh(z) + \sinh(z),\nonu\\
&\nonu\\
&D:= f^2 + g^2,\quad \text{where}\quad f, g~~\text{and derivatives are given by:}\nonu\\
\end{align}


\begin{align}
&  g=\frac{\bt\sqrt{\al^2+\bt^2}}{\al\sqrt{\Delta}}\sin(\al y_1) -  \frac{\sqrt{2}\bt\mu e^{\bt y_2}}{\Delta},\\
&  g_1:=g_x=\frac{\bt\sqrt{\al^2+\bt^2}}{\sqrt{\Delta}}\cos(\al y_1) -  \frac{\sqrt{2}\bt^2\mu e^{\bt y_2}}{\Delta},
\end{align}
\begin{align}
&  g_2:=g_t=\frac{\bt\delta\sqrt{\al^2+\bt^2}}{\sqrt{\Delta}}\cos(\al y_1) -  \frac{\sqrt{2}\bt^2\ga\mu e^{\bt y_2}}{\Delta},\\
&  g_3:=g_{xt}=-\frac{\al\bt\delta\sqrt{\al^2+\bt^2}}{\sqrt{\Delta}}\sin(\al y_1) -  \frac{\sqrt{2}\bt^3\ga\mu e^{\bt y_2}}{\Delta},
\end{align}
\begin{align}
&  g_4:=g_{xx}=-\frac{\al\bt\sqrt{\al^2+\bt^2}}{\sqrt{\Delta}}\sin(\al y_1) -  \frac{\sqrt{2}\bt^3\mu e^{\bt y_2}}{\Delta},\\
&  g_5:=g_{xxt}=-\frac{\al^2\bt\delta\sqrt{\al^2+\bt^2}}{\sqrt{\Delta}}\cos(\al y_1) -  \frac{\sqrt{2}\bt^4\ga\mu e^{\bt y_2}}{\Delta},\\
&\nonu\\
&  f=\cosh(\bt y_2)-\frac{\sqrt{2}\bt\mu}{\al\sqrt{\Delta}}\cos(\al y_1+\arctan(\bt/\al)),\\
&  f_1:=f_x=\bt\sinh(\bt y_2)+\frac{\sqrt{2}\bt\mu}{\sqrt{\Delta}}\sin(\al y_1+\arctan(\bt/\al)),
\end{align}
\begin{align}
&  f_2:=f_t=\bt\ga\sinh(\bt y_2)+\frac{\sqrt{2}\bt\delta\mu}{\sqrt{\Delta}}\sin(\al y_1+\arctan(\bt/\al)),\\
&  f_3:=f_{xt}=\bt^2\ga\cosh(\bt y_2)+\frac{\sqrt{2}\al\bt\delta\mu}{\sqrt{\Delta}}\cos(\al y_1+\arctan(\bt/\al)),\\
&  f_4:=f_{xx}=\bt^2\cosh(\bt y_2)+\frac{\sqrt{2}\al\bt\mu}{\sqrt{\Delta}}\cos(\al y_1+\arctan(\bt/\al)),\\
&  f_5:=f_{xxt}=\bt^3\ga\sinh(\bt y_2)-\frac{\sqrt{2}\al^2\bt\delta\mu}{\sqrt{\Delta}}\sin(\al y_1+\arctan(\bt/\al)).\label{notacionFG}
 \end{align}
 \noindent

From the explicit expression of the mKdV breather with NVBC \eqref{BmksdNVBC} but now written in terms of the above derivatives (B.2)-\eqref{notacionFG}, 
we obtain that:

\begin{align}\label{Btexp}
&B   = \mu + 2\sqrt{2}\frac{fg_1-f_1g}{D} 
\qquad\text{and}\qquad \tilde B_t  = 2\sqrt{2}\frac{fg_2-f_2g}{D}.
\end{align}

Moreover, from  \eqref{massBnvbc}, we also have an equivalent expression for the quantity $(\mathcal M_{\al,\bt})_t $ 
\begin{align}\label{expandMt}
&(\mathcal M_{\al,\bt})_t  = -\frac{4}{D^2}(gg_1 + ff_1)(gg_2 + ff_2) + \frac{2}{D}\Big(gg_3 + g_1g_2 + ff_3 + f_1f_2\Big),
\end{align}
\ms

and therefore from \eqref{Btexp} and \eqref{expandMt},


\begin{align}\label{2MtB}
&2(\mathcal{M}_{\al,\bt})_tB  = \frac{N_2}{f^2D^3},
\end{align}
\noindent
where
\begin{align}\label{N2}
&N_2:=-f^2\Big(16\sqrt{2}(g_1f-gf_1)(ff_1+gg_1)(ff_2 + gg_2)\nonu\\
&-8\sqrt{2}D(g_1f-gf_1)(gg_3 + g_1g_2 + ff_3 + f_1f_2)\nonu\\
&+8\mu D(f_1f+gg_1)(ff_2 + gg_2)- 4D^2\mu(gg_3 + g_1g_2 + ff_3 + f_1f_2)\Big).
\end{align}

Now, we compute $B_{xt}$. First we get

\begin{align}
&B_x = -\frac{4\sqrt{2}}{fD^2}\Big((gfg_1-g^2f_1)(fg_1-gf_1)\Big) + \frac{2\sqrt{2}}{fD}\Big(f^2g_4 - 2g_1ff_1 + 2gf_1^2-gff_4\Big),\label{Bx}
\end{align}
\noindent
and then

\begin{align}\label{Bxt}
&B_{xt} = 4\sqrt{2}\frac{N_1}{f^2D^3},
\end{align}
\noindent
where 
\begin{align}\label{N1}
&N_1:=\Big(4g^2(g_1f-gf_1)^2(g_2f-gf_2)\nonu\\
&-Dg(g_1f-gf_1)(g_3f^2-g_2ff_1-g_1ff_2+2gf_1f_2-gff_3)\nonu\\
&-D(g_1f-gf_1)(g_1g_2f^2+gg_3f^2-2gg_2ff_1-2gg_1ff_2+3g^2f_1f_2-g^2ff_3)\nonu\\
&-Dg(g_2f-gf_2)(g_4f^2-2g_1ff_1+2gf_1^2-gff_4)\nonu\\
&+D^2\Big(\frac{1}{2}(g_5f^3-g_4f^2f_2-g_2f^2f_4-gf^2f_5-2g_3f^2f_1)\nonu\\
&+g_2ff_1^2+2g_1ff_1f_2-3gf_1^2f_2-g_1f^2f_3+2gff_1f_3+gff_2f_4\Big)\Big).
\end{align}

Now, we verify by using the symbolic software \emph{Mathematica} that, after expanding $f's$ and $g's$ terms (B.2)-\eqref{notacionFG} and lengthy rearrangements, 
the sum $N_1+N_2$  simplifies as follows:

\begin{align}\label{N1N2}
&N_1 + N_2 =f^2D^2\cdot A_1 2\sqrt{2}(g_2f-gf_2) + f^2D^2\cdot A_2 2\sqrt{2}(g_1f-gf_1).\\
\end{align}

Finally, remembering \eqref{Btexp}, we have that:

\begin{align}\label{bxtmtb}
&B_{xt} + 2(\mathcal{M}_{\al,\bt})_tB = \frac{N_1 + N_2}{f^2D^3}  \nonu\\
&=\frac{f^2D^2\cdot A_1 2\sqrt{2}(g_2f-gf_2)}{f^2D^3} + \frac{f^2D^2\cdot A_2 2\sqrt{2}(g_1f-gf_1)}{f^2D^3}\nonu\\
&=A_12\sqrt{2}\frac{(g_2f-gf_2)}{D} + A_22\sqrt{2}\frac{(g_1f-gf_1)}{D}=A_1 \tilde{B}_t + A_2(B-\mu).
\end{align}

\end{document}